\documentclass[12pt]{article}
\usepackage{amsmath}
\usepackage{tikz} 
\usetikzlibrary{arrows.meta}
\usepackage{tikz}
\usetikzlibrary{arrows.meta,backgrounds,fit}
\usepackage{amsmath}
\usetikzlibrary{arrows.meta, bending, calc}
\usepackage{amssymb,amsfonts,amsmath}
\usepackage{float}
\usepackage{booktabs}
\usepackage{adjustbox}

\usepackage{array}
\usepackage{graphicx}
\usepackage{multirow}
\usepackage{color}
\def\blue{\textcolor{blue}}

\def\blue{\textcolor{blue}}
\usepackage{amsfonts, amsthm, amsmath}
\allowdisplaybreaks[4]
\usepackage{pgfplots}
\usepackage{xcolor}
\usepackage{afterpage}
\usepackage{rotating}

\usepackage{tikz}
\usepackage{lmodern}

\usetikzlibrary{decorations.pathmorphing}
\tikzset{
	redline/.style={red, thick, ->, shorten >=1pt, shorten <=1pt}
}

\usepackage{xcolor}
\usepackage{framed}
\usepackage{graphics}

\usepackage{amssymb}

\usepackage{amscd}

\usepackage{t1enc}

\usepackage[mathscr]{eucal}

\usepackage{indentfirst}

\usepackage{enumitem}

\usepackage{enumitem}

\usepackage{graphicx}

\usepackage{graphics}

\usepackage{pict2e}

\usepackage{mathrsfs}

\usepackage{comment}

\usepackage{enumitem}

\usepackage{enumerate}

\usepackage{hyperref}

\hypersetup{colorlinks=true}

\usepackage{color}
\usepackage{epic}
\usepackage{framed}
\makeatletter
\IfFileExists{mathabx.sty}{}{%
	\expandafter\def\csname ver@mathabx.sty\endcsname{2026/08/02 fallback}%
}
\makeatother
\usepackage{mathabx}
\usepackage{booktabs}
\usetikzlibrary{decorations.pathreplacing,calc}
\usepackage{makecell}
\newcolumntype{V}{!{\vrule width 2pt}}

\numberwithin{equation}{section}

\def\blue{\textcolor{blue}}

\def\blue{\textcolor{blue}}

\theoremstyle{plain}
\newcommand{\rmnum}[1]{\romannumeral #1}
\newtheorem{theorem}{Theorem}[section]

\newtheorem{conjecture}[theorem]{Conjecture}

\newtheorem{lemma}[theorem]{Lemma}
\newtheorem{definition}[theorem]{Definition}

\newtheorem{problem}[theorem]{Problem}
\definecolor{handred}{RGB}{220, 0, 0}
\definecolor{handgreen}{RGB}{0, 150, 0}

\def\des{\mathsf{des}}

\def\orc{\mathsf{orc}}
\def\Alt{\mathsf{Alt}}

\def\Left{\mathsf{left}}
\def\st{\mathsf{st}}
\def\larm{\mathsf{larm}}

\def\l{\mathsf{left}}

\def\fleaf{\mathsf{fleaf}}
\def\sleaf{\mathsf{sleaf}}

\def\asc{\mathsf{asc}}
\def\B{\mathsf{B}}
\def\AB{\mathsf{AB}}

\def\Ascbot{\mathsf{Ascbot}}
\def\Desbot{\mathsf{Desbot}}

\def\JB{\mathsf{JB}}

\def\cpen{\mathsf{cpen}}
\providecommand{\spi}{\mathsf{spi}}
\providecommand{\emptyword}{\epsilon}
\def\first{\mathsf{first}}
\def\last{\mathsf{last}}
\def\Lefttop{\mathsf{Lefttop}}

\def\sn{\mathrm{sn}}
\def\cn{\mathrm{cn}}
\def\dn{\mathrm{dn}}
\newcommand{\der}{{\rm{d}}}
\def\S{\mathfrak{S}}

\begin{document}
	\begin{center}
		{\Large\bf Bijective proofs of several conjectures on\\ Jacobi permutations}
	\end{center}
	
	\begin{center}
		{\small Zhicong Lin$^a$,  Yongzhou Wen$^{b,*}$, Sherry H.F. Yan$^b$}\\[4pt]
		$^a$Research Center for Mathematics and Interdisciplinary Sciences,  Shandong University\\
		\& Frontiers Science Center for Nonlinear Expectations, Ministry of Education \\
		Qingdao 266237, P.R. China\\[2pt]
		$^b$ Department of Mathematics, Zhejiang Normal University\\
		Jinhua 321004, P.R. China
	\end{center}
	\footnote{$^*$Corresponding author. \\{\em E-mail address:} linz@sdu.edu.cn(Z. Lin),    xxgz771@zjnu.edu.cn(Y. Wen), hfy@zjnu.cn (S.H.F. Yan).}
	
	\vskip0.1in
	\noindent {\bf Abstract.}
	Jacobi permutations,   invented  by Viennot in the context of  the Jacobi elliptic functions, are  counted by the Euler numbers.  
	Recently, Henke, Hoffman, Stephens, Yuan, and Zhuang studied refined enumerations of Jacobi permutations and proposed three conjectures concerning the distribution of several statistics  on  Jacobi permutations. In this paper, we prove these conjectures by establishing explicit bijections involving  increasing even trees,  increasing binary trees, alternating permutations, and André permutations.   One
 highlight of our results is  a bijection between Jacobi permutations and André I permutations that transforms the pair of statistics  $(\Ascbot, \last)$ to  the pair of statistics $(\Desbot, \first)$. Here  the statistic $\Ascbot$ (resp., $\Desbot$) denotes the set of ascent bottoms (resp., descent bottoms) of permutations,  and the statistic $\first$ (resp., $\last$) denotes the first (resp., $\last$) letter of  permutations.  Furthermore, we investigate pairs of statistics on Andr\'e permutations and simsun permutations that are equidistributed with  the pair $(\asc, \last)$ on Jacobi permutations, where $\asc$ denotes the number of ascents of permutations.  Finally, we obtain a closed-form formula for the trivariate exponential generating function of Jacobi permutations with respect to the number of ascents and the numbers of letters smaller and larger than the last letter. 
	

	\section{Introduction}
	 The \blue{\em Jacobi elliptic function} $\sn(u,\alpha)$ may be (see~\cite{Flajolet1989,Vie80}) defined by the inverse of an elliptic integral:
$$
\sn(u,\alpha)=y\quad\text{iff}\quad u=\int_0^y\frac{\der t}{\sqrt{(1-t^2)(1-\alpha^2t^2)}},
$$
where $\alpha\in(0,1)$ is a real number. The other two functions are given by
$$
\cn(u,\alpha)=\sqrt{1-\sn^2(u,\alpha)}\quad\text{and}\quad \dn(u,\alpha)=\sqrt{1-\alpha^2\sn^2(u,\alpha)}.
$$
  The Jacobi permutations were   introduced   by Viennot~\cite{Vie80}  to
interpret the coeﬃcients in the Taylor expansion of the Jacobi elliptic functions $\sn$, $\cn$, and $\dn$. Jacobi permutations of length $n$ are  counted by the \blue{\em Euler number}
  $E_n$ appearing in the expansion 
  $$
  \tan(x)+\sec(x)=\sum_{n\geq0}E_n\frac{x^n}{n!}=1+x+1\frac{x^2}{2!}+2\frac{x^3}{3!}+5\frac{x^4}{4!}+16\frac{x^5}{5!}+61\frac{x^6}{6!}+\cdots.
  $$
  The main objective of this paper is to resolve three enumerative conjectures and one open problem regarding Jacobi permutations posed by Henke-Hoffman-Stephens-Yuan- Zhuang~\cite{HHSSY25}.
  
We begin with several important classes of permutations counted by Euler numbers.  Let $S$ be  a finite set  of positive integers with $|S|=n$.  A  permutation $\pi=\pi_1\pi_2\ldots \pi_n$ of $S$ is a linear arrangement of the elements of  $S$.  Let $\mathfrak{S}_S$ denote the set of permutations of $S$. When $S=[n]:=\{1,2,\ldots, n\}$, we abbreviate $\S_S$ as  $\S_n$. 
A permutation $\pi\in\S_S$  is said  to be   \blue{\em alternating  (or up-down) } if 
   $$\pi_1<\pi_2>\pi_3<\pi_4>\cdots.$$   
   Denote by $\Alt_n$ the set of all alternating  permutations in $\S_n$. For example,  we have 
   $\Alt_4=\{1423, 1324, 2314, 2413, 3412\}$.  A fundamental result attributed to Andr\'e~\cite{And79} asserts  that $E_n=|\Alt_n|$, furnishing the first combinatorial interpretation for Euler numbers.   Let $\Alt_{n,k}$ denote the  set of   permutations $\pi\in \Alt_n$ with $\pi_1=k$.    
   As a refinement of the Euler number, the \blue{\em Entringer number}  $E_{n,k}$ \cite{Ent66} is defined  as the cardinality of the set $\Alt_{n,k}$.  It can be shown bijectively that the  Entringer number $E_{n,k}$ satisfies the  recurrence relation:
   $$
   E_{n,k}=E_{n,k+1}+E_{n-1, n-k}
   $$
   for all $n\geq 2$ and $1\leq k<n$.   
   Entringer numbers encode the distribution of several statistics on some other classes of permutations enumerated by Euler numbers, including Jacobi permutations~\cite{HHSSY25}, André permutations~\cite{FH16}, and simsun permutations~\cite{Lin}.

  Let $\min(\pi)$ and $\max(\pi)$ denote the minimum letter and the maximum letter of $\pi$, respectively.    
 Given a permutation  $\pi\in\S_S$ and  a letter $x$ of $\pi$, let $\rho_x(\pi)$ be the maximal consecutive
 subword of $\pi$ consisting of letters immediately to the right of $x$ which are all larger than $x$.  Then $\pi$ is said to be a \blue{\em Jacobi permutation} if and only if $|\rho_x(\pi)|$ is even for each  letter $x$ of $\pi$. Note that the empty permutation is always regarded as a  Jacobi permutation. 
  Jacobi permutations can also be defined recursively as follows.
  	Consider the decomposition of a nonempty permutation $\pi$ as   $\pi=\alpha \min(\pi)\beta$, where $\alpha$ and $\beta$ are (possibly empty) permutations.   Then $\pi$ 
  	is said to be a  Jacobi permutation if  both $\alpha$ and $\beta$ are Jacobi permutations and $|\beta|$ is even. 
   Let $\mathcal{J}_S$ denote the set of all Jacobi permutations of $S$. When $S=[n]$,  we abbreviate $\mathcal{J}_S$ as  $\mathcal{J}_n$. 
   For instance,  $\mathcal{J}_4=\{ 
   4321,  2431, 4132, 3142,  2143\}$. 
   

   Andr\'e permutations, introduced by Foata and Sch\"utzenberger \cite{FS70} and further studied by Strehl \cite{Str74} and Foata and Strehl \cite{FS74, FS76}, provide additional combinatorial interpretations for the Euler numbers.  The empty
   permutation  and any single-letter permutation are defined as both Andr\'e   $\mathrm{I} $ permutations 
   and Andr\'e  $\mathrm{II} $ permutations.
   For a permutation $\pi= \pi_1\pi_2\cdots\pi_n$ of  $S$ with $n\ge 1$, consider its decomposition $\pi= \alpha \min(\pi)\beta$, where  $\alpha$ and $\beta$ are (possibly empty) permutations. Then $\pi$ is said to be an \blue{\em Andr\'e  $\mathrm{I} $ permutation}   if both $\alpha$ and $\beta$ are Andr\'{e} $\mathrm{I} $ permutations  and  the maximum letter  of the concatenated subword $\alpha\beta$ belongs to $\beta$, i.e., $\max(\alpha\beta)=\max(\beta)$.  Similarly, $\pi$ is said to be an \blue{\em Andr\'e  $\mathrm{II} $ permutation}   if   both $\alpha$ and $\beta$ are Andr\'{e} $\mathrm{II} $ permutations  and  the minimum  letter  of the concatenated subword $\alpha\beta$ belongs to $\beta$, i.e., $\min(\alpha\beta)=\min(\beta)$.
   Let $\mathcal{A}^{\mathrm{I}}_S$  and $\mathcal{A}^{\mathrm{II}}_S$ denote the set of all Andr\'{e} $\mathrm{I} $ permutations and Andr\'{e} $\mathrm{II} $ permutations of $S$, respectively. When $S=[n]$,  we abbreviate $\mathcal{A}^{\mathrm{I}}_S$ and $\mathcal{A}^{\mathrm{II}}_S$ as  $\mathcal{A}^{\mathrm{I}}_n$ and $\mathcal{A}^{\mathrm{II}}_n$, respectively. 
   For instance,  $\mathcal{A}^{\mathrm{I}}_4=\{1234, 1324, 2314, 2134, 3124
    \}$ and  $\mathcal{A}^{\mathrm{II}}_4=\{ 1234, 1423, 3412, 4123, 3124  
    \}$. 
    
    Simsun permutations were introduced by   Simion and  Sundaram in a series of studies of homology representations of the symmetric
    group  \cite{Su94, Su95, Su96}.  A permutation $\pi\in \mathfrak{S}_n$ is called a  \blue{\em simsun permutation} if $\pi$  contains no double descents, and this property is preserved after removing the elements
    $n,n-1, \ldots, 1$ in order.
    Recall that an index $i$ with $1<i<n$ is  called a {\em double descent} of a permutation $\pi=\pi_1\pi_2\ldots \pi_n$  if $\pi_{i-1}>\pi_i>\pi_{i+1}$. For instance, the permutation 
    $\pi=892714365$ is a simsun permutation since the permutations $892714365$, $82714365$, $2714365$, $214365$, $21435$, $2143$, $213$, $21$, $1$ have no double descents. 
     For $n\geq 1$, let $\mathcal{RS}_{n}$ denote the set of simsun permutations in $\mathfrak{S}_n$.  For instance,  $\mathcal{RS}_3=\{123, 132, 213, 231, 312\}$.  A remarkable property of simsun permutations is that $|\mathcal{RS}_n|= E_{n+1}$. 
    In \cite{CS11}, Chow and  Shiu
    proved  that the number of descents are equidistributed
    over Andr\'e $\mathrm{I}$ permutations and simsun permutations.

   Given a permutation $\pi=\pi_1\pi_2\cdots \pi_n$,  define the following permutation statistics
   \begin{align*}
   	&\last(\pi)=\pi_n, \quad \first(\pi)=\pi_1,\\
   	& \Ascbot(\pi)=\{\pi_i\mid \pi_i<\pi_{i+1}, 1\leq i<n\}, \\
   	& \Desbot(\pi)=\{\pi_i\mid \pi_{i-1}>\pi_{i}, 1< i\leq n\},\\
	&\asc(\pi)=|\Ascbot(\pi)|\quad\text{and}\quad\des(\pi)=|\Desbot(\pi)|.
   	\end{align*}
    For $n\geq 2$, let $\cpen(\pi):=n+1-\pi_{n-1}$. 
Let
   $$
   \begin{array}{ll}
   		\vspace{8pt}
   	& j^{\last}_{n,k}:=|\{\pi\in \mathcal{J}_n\mid \last(\pi)=k\}|,\\
   	\vspace{8pt}
  & j^{\cpen}_{n,k}:=|\{\pi\in \mathcal{J}_n\mid \cpen(\pi)=k\}|,\\
 \vspace{8pt}
  & j^{(\cpen, \last)}_{n,i,k}:=|\{\pi\in \mathcal{J}_n\mid \cpen(\pi)=i, \last(\pi)=k\}|,\\
  \vspace{8pt}
  & j^{\asc}_{n,k}:=|\{\pi\in \mathcal{J}_n\mid \asc(\pi)=k\}|,\\
  \vspace{8pt}
   & j^{(\asc, \last)}_{n,i,k}:=|\{\pi\in \mathcal{J}_n\mid \asc(\pi)=i, \last(\pi)=k\}|,\\
    \vspace{8pt}
   & a^{\des}_{n,k}:=|\{\pi\in \mathcal{A}^{\mathrm{I}}_n\mid \des(\pi)=k\}|,\\
    \vspace{8pt}
  & a^{(\des, \first)}_{n,i,k}:=|\{\pi\in \mathcal{A}^{\mathrm{I}}_n\mid \des(\pi)=i, \first(\pi)=k\}|.\\
   \end{array}
   $$
   Very recently,  Henke-Hoffman-Stephens-Yuan-Zhuang \cite{HHSSY25}  proved that 
   \begin{equation}\label{eq-J-last}
   	E_{n,k}=j^{\last}_{n,k}\quad\text{and}\quad 
    j^{\asc}_{n,k}=a^{\des}_{n,k}.
    \end{equation} They 
   further   posed the following three conjectures concerning the distribution of the statistics  on  Jacobi permutations. 
   
   \begin{conjecture}[\text{\cite[Conjecture~9.3]{HHSSY25}}] \label{con1} For all $n \ge 2$ and $k \ge 1$, we have $j^{\cpen}_{n,k} = E_{n,k}$.
   	\end{conjecture}

    \begin{conjecture}[\text{\cite[Conjecture~9.4]{HHSSY25}}]\label{con2}
     For all $n \ge 2$ and $i,k \ge 1$, we have $j^{(\cpen, \last)}_{n,i, k} = E_{n-1, n+1-i-k}$.
   \end{conjecture}

    \begin{conjecture}[\text{\cite[Conjecture~9.5]{HHSSY25}}]\label{con3} For all $n \ge 1$, $i\geq 0,$ and $k \ge 1$, we have $j^{(\asc, \last)}_{n,i, k} = a^{(\des, \first)}_{n,i,k}$.
   \end{conjecture}
   
   In this paper, we provide a bijective proof of Conjecture~\ref{con2} by  employing a  bijection between  alternating permutations and increasing even trees constructed by Kuznetsov, Pak, and Postnikov~\cite{KPP}. 
  As observed in \cite{HHSSY25},  Conjecture~\ref{con2} asserts that  the statistics $\cpen$ and $\last$ have  a symmetric joint distribution on $\mathcal{J}_n$. Combined with the first equality in~\eqref{eq-J-last}, Conjecture~\ref{con2} further implies Conjecture~\ref{con1}.
  
  For a set $A$ of positive integers, write $q^{A}=\prod_{i\in A}q_i$ with the convention that $q^{\emptyset}=1$. 
  Our second  main result is concerned with the following strengthened form  of Conjecture~\ref{con3}.
  \begin{theorem}\label{thm-main}
  		Fix a  set $S$ of positive integers.  There exists a bijection  $\Omega$  between $\mathcal{J}_S$ and $\mathcal{A}^{\mathrm{I}}_S$  such that for any 
  		$\pi\in \mathcal{J}_S$, we have
  		\begin{equation}\label{eq-omega}
  			(\Ascbot,   \last)\pi= (\Desbot,   \first)\Omega(\pi).
  			\end{equation}
  		Consequently, 
  		\begin{equation}\label{eq-Omega}
  			\sum\limits_{\pi\in \mathcal{J}_n}q^{\Ascbot(\pi)}t^{\last(\pi)}=\sum\limits_{\pi\in \mathcal{A}^{\mathrm{I}}_n}q^{\Desbot(\pi)} t^{\first(\pi)}.
  			\end{equation}
  	\end{theorem}

   A pair of statistics $(\st_1, \st_2)$   is said to be \blue{\em André–Entringer} if $(\st_1, \st_2)$ on certain objects counted by Euler numbers has the same distribution as $(\des, \first)$ on Andr\'e I permutations. Theorem~\ref{thm-main} implies that $(\asc,\last)$ is a pair of André–Entringer statistics on Jacobi permutations, as was conjectured in~\cite{HHSSY25}. 
  	Our third main result is concerned with the following  André–Entringer statisitics on Andr\'e II permutations and simsun permutations.
   
  \begin{theorem}\label{thm-main2}
  	For $n\geq 1$, we have 
  	\begin{equation}\label{eq-equidistribution}
\sum_{\pi\in \mathcal{A}^{\mathcal{I}}_{n}}q^{\des(\pi)} t^{\first(\pi)} =	\sum_{\pi\in \mathcal{A}^{\mathcal{II}}_{n}}q^{\des(\pi)} t^{n+1-\last(\pi)}=\sum_{\pi\in \mathcal{RS}_{n-1}} q^{\des(\pi)}t^{n-\last(\pi)}.
  	\end{equation}
  	\end{theorem}
  	Set $D_0(t):=1$ and  $D_n(t):=\sum_{\pi\in\mathcal J_n}t^{\operatorname{asc}(\pi)}$ for $n\geq1$. 
Introduce
  \[
  	D(t;z):=\sum_{n\geq0}D_n(t)\frac{z^n}{n!},
  	\]
 the bivariate exponential generating function for Jacobi permutations with respect to the number of ascents.
  	In \cite{HHSSY25}, Henke-Hoffman-Stephens-Yuan-Zhuang proved that 
  	\begin{equation}
  		D(t;z)
  		=
  		\frac{(\rho-1)+(1+\rho)e^{\rho z}}
  		{(1+\rho)+(\rho-1)e^{\rho z}},
  		\label{eq:jacobi-D-closed}
  	\end{equation}
  	where $\rho=\sqrt{1-2t}$. 
    The generating function $D(t,z)$ has appeared in an equivalent form in the study of Andr\'e
    permutations.  Foata and Sch{\"u}tzenberger  \cite{foata1973nombres} proved that
    $$
    \sum_{n\geq 0}\sum\limits_{\pi\in \mathcal{A}^{\mathrm{I}}_n}t^{\des(\pi)+1}{z^n\over n!}={\rho(1+we^{\rho z})\over 1-we^{\rho z}},
    $$
  where   $w={1-\rho\over 1+\rho}$.  Since $$D(t,z)-1={1\over t}\big({\rho(1+we^{\rho z})\over 1-we^{\rho z}}-1\big),$$  the second equality in~\eqref{eq-J-last} was established in~\cite{HHSSY25}. They further posed the following refined enumeration problem. 
  
  \begin{problem}[\text{\cite[Problem~9.6]{HHSSY25}}]
  \label{problem:1}
  Find a formula for computing the Andr\'e–Entringer distribution.
  \end{problem}

  For $1\leq k\leq n$, define 
  \[
  J_{n,k}(t)
  :=
  \sum_{\substack{\pi\in\mathcal J_n\\
  		\operatorname{last}(\pi)=k}}
  t^{\operatorname{asc}(\pi)}.
  \]
  Let
$$
  	F(t;x,y):=\sum_{n\geq1}\sum_{k=1}^{n}
       J_{n,k}(t)
  	\frac{x^{k-1}}{(k-1)!}
  	\frac{y^{n-k}}{(n-k)!}
$$
  be the trivariate exponential generating function of Jacobi permutations with respect to the number of ascents and the numbers of letters smaller and larger than the last letter. 
 Finally, we obtain the following closed-form formula for $F(t;x,y)$, which solves Problem~\ref{problem:1}. 
  \begin{theorem}
  	\label{thm:jacobi-F-closed}
  	Let $\rho=\sqrt{1-2t}$.
  	Then
  	\begin{equation}
  		F(t;x,y)
  		=
  		\frac{(1+\rho)e^{\rho y}+(\rho-1)e^{\rho x}}
  		{(1+\rho)+(\rho-1)e^{\rho(x+y)}}.
  		\label{eq:jacobi-F-closed}
  	\end{equation}
  \end{theorem}
  
  {\bf The rest of this paper is organized as follows.} In Sec.~\ref{Sec:2}, we provide a  proof of Conjecture~\ref{con2} by employing a  bijection between  alternating permutations and increasing even  trees constructed by Kuznetsov-Pak-Postnikov~\cite{KPP}. In Sec.~\ref{Sec:3}, we construct the bijection $\Omega$, leading to a bijective proof of Theorem~\ref{thm-main}. Sec.~\ref{Sec:4} is devoted to the bijective proof of Theorem \ref{thm-main2}, while Sec.~\ref{Sec:5} aims to prove Theorem~\ref{thm:jacobi-F-closed}. 
  
  \section{Proof of Conjecture~\ref{con2}}
  \label{Sec:2}
  This section is devoted to the bijective proof of Conjecture~\ref{con2}.  This is accomplished by employing a  bijection between  alternating permutations and increasing even  trees constructed by Kuznetsov-Pak-Postnikov~\cite{KPP}. 
  Let us first review some terminology related to trees.   An \blue{\em ordered tree} is  a tree with one designated node, which is called the root, and the subtrees of each node are linearly ordered. In this paper we will assume that all the trees are ordered. 
  \begin{definition}[Increasing trees]
  	Fix a  set $S$ of positive integers. An
  	\blue{ increasing tree} on $S$ is an ordered tree with $|S|+1$ nodes that are labeled precisely by all elements in the set $S\cup\{0\}$ satisfying
  	\begin{itemize} 
  		\item  the labels along a path from the root to any leaf are  increasing;
  		\item  the labels of the children of each node are   increasing from left to right.
  	\end{itemize}
  	Denote by  $\mathcal{T}_{S}$ the set of all  increasing trees on  $S$.
  \end{definition}
  
  \begin{definition}[increasing even trees]
  	Fix a  set $S$ of positive integers. An
  	\blue{   increasing even tree} on $S$  is an increasing tree in $\mathcal{T}_S$  in which every internal node other than the root  has an even number of children. 
  	See Fig.~\ref{fig:increasing-even-tree-11}  for an example of  an increasing even tree.  Denote by  $\mathcal{ET}_{S}$ the set of all  even increasing trees on  $S$.
  \end{definition}

  There exists    a natural bijection $\kappa$ between $\mathfrak{S}_S$ and $\mathcal{T}_S$ (see~\cite[Chapter 1]{ST}). Given  $\pi\in \mathfrak{S}_S$, construct the increasing tree $\kappa(\pi)$ on $S$ by letting   $i$  be the child of the
  rightmost element $j$, $j< i$, of $\pi$ which precedes $i$. If there is no such element $j$, then let $i$ be the
  child of the root $0$. See Fig.~\ref{fig:increasing-even-tree-11} for  the increasing tree $\kappa(4(10)731865(11)92)$.  Moreover, it restricts to a bijection between $\mathcal{J}_S$ and  $\mathcal{ET}_S$ \cite{Vie80}.  
  
  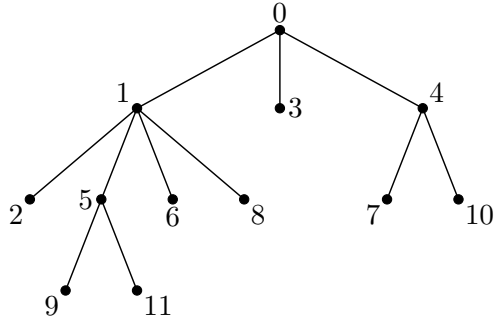
\begin{figure}
  	\centering
  	
  	\begin{tikzpicture}[
  		scale=0.7,
  		x=0.90cm,
  		y=0.82cm,
  		tree vertex/.style={
  			circle,
  			fill=black,
  			draw=black,
  			inner sep=0pt,
  			minimum size=3.5pt
  		},
  		every label/.style={
  			font=\small,
  			text=black,
  			inner sep=0.9pt
  		},
  		normal edge/.style={
  			draw=black,
  			line width=0.55pt,
  			line cap=round
  		},
  		transformation arrow/.style={
  			->,
  			draw=black,
  			line width=0.55pt,
  			>=stealth
  		},
  		tree name/.style={
  			font=\normalsize
  		},
  		%
  		even tree/.style={
  			shift={(0.00,0.00)}
  		}
  		]
  		
  		\begin{scope}[even tree]
  			\coordinate (v0)  at ( 0.00, 0.00);
  			
  			\coordinate (v1)  at (-3.00,-1.80);
  			\coordinate (v3)  at ( 0.00,-1.80);
  			\coordinate (v4)  at ( 3.00,-1.80);
  			
  			\coordinate (v2)  at (-5.25,-3.90);
  			\coordinate (v5)  at (-3.75,-3.90);
  			\coordinate (v6)  at (-2.25,-3.90);
  			\coordinate (v8)  at (-0.75,-3.90);
  			
  			\coordinate (v7)  at ( 2.25,-3.90);
  			\coordinate (v10) at ( 3.75,-3.90);
  			
  			\coordinate (v9)  at (-4.50,-6.00);
  			\coordinate (v11) at (-3.00,-6.00);
  			
  			\draw[normal edge]
  			(v0)--(v1)
  			(v0)--(v3)
  			(v0)--(v4)
  			(v1)--(v2)
  			(v1)--(v5)
  			(v1)--(v6)
  			(v1)--(v8)
  			(v5)--(v9)
  			(v5)--(v11)
  			(v4)--(v7)
  			(v4)--(v10);
  			
  			\node[tree vertex,label=above:{0}]       at (v0)  {};
  			
  			\node[tree vertex,label=above left:{1}]  at (v1)  {};
  			\node[tree vertex,label=right:{3}]       at (v3)  {};
  			\node[tree vertex,label=above right:{4}] at (v4)  {};
  			
  			\node[tree vertex,label=below left:{2}]  at (v2)  {};
  			\node[tree vertex,label=left:{5}]        at (v5)  {};
  			\node[tree vertex,label=below:{6}]       at (v6)  {};
  			\node[tree vertex,label=below right:{8}] at (v8)  {};
  			
  			\node[tree vertex,label=below left:{7}]  at (v7)  {};
  			\node[tree vertex,label=below right:{10}] at (v10) {};
  			
  			\node[tree vertex,label=below left:{9}]  at (v9)  {};
  			\node[tree vertex,label=below right:{11}] at (v11) {};
  		\end{scope}
  		
  		
  		\begin{scope}[even tree]
  		\end{scope}
  		
  	\end{tikzpicture}
  	
  	\caption{An increasing even tree.}
  	\label{fig:increasing-even-tree-11}
  \end{figure}
  
   The \blue{\em preorder  traversal} of an ordered tree proceeds by visiting the root, then recursively performing  the  preorder  traversal  on each subtree of the root from left to right. 
   Given an increasing tree $T$, let $\fleaf(T)$ and $\sleaf(T)$ denote the labels of the first leaf and the second leaf visited under the preorder  traversal of $T$, respectively. 
  For example, if we let $T$ be the tree displayed in Fig.~\ref{fig:increasing-even-tree-11}, then $\fleaf(T)=2$ and $\sleaf(T)=9$. 
  The path from the root to the first leaf is called the \blue{\em leftmost path} of $T$.

  \begin{lemma}\label{lem-kappa}
  	Fix a set $S$ of positive  integers with $|S|\geq 2$. The bijection $\kappa$ induces a one-to-one correspondence between $\mathcal{J}_S$ and $\mathcal{ET}_{S}$ such that for any $\pi\in \mathcal{J}_S$, we have 
  	\begin{equation}\label{eq-kappa}
  		\last(\pi)=\fleaf(\kappa(\pi))\quad\text{and}\quad\cpen(\pi)=n+1-\sleaf(\kappa(\pi)).
  	\end{equation}
  \end{lemma}
  \begin{proof}
  	It remains to verify that (\ref{eq-kappa}) holds for  any $\pi\in \mathcal{J}_S$.  	Let $\pi=\pi_1\pi_2\ldots \pi_n$ be a Jacobi permutation of $S$. In the following, we aim to show  that 	
	\begin{equation}\label{eq: fsleaf}
	\fleaf(\kappa(\pi))=\pi_n\quad\text{and}\quad \sleaf(\kappa(\pi))=\pi_{n-1}.
	\end{equation}

  	By the construction of $\kappa(\pi)$, the node $\pi_n$ must be the first leaf in $\kappa(\pi)$.  
  	Let   $\pi'=\pi_1\pi_2\ldots \pi_{n-1}$. Again by the construction of $\kappa(\pi)$ and $\kappa(\pi')$, one can easily check  $\kappa(\pi)$ is obtained from $\kappa(\pi')$  by attaching a   leaf $\pi_n$ to a node on the leftmost path  in $\kappa(\pi')$ and the node  $\pi_{n-1}$ is the first leaf in $\kappa(\pi')$.     Since $\pi$ is a Jacobi permutation, we must have  $\pi_{n-1}>\pi_n$. This   implies that the node $\pi_{n-1}$ must also be a leaf in  $\kappa(\pi)$.   Hence, we conclude that Eq.~\eqref{eq: fsleaf} holds, which completes the proof.  
  \end{proof}
  
  Let $\mathcal{Z}_n$ denote set of sequences $(z_1, z_2, \ldots, z_n)$ such that
   $1\leq z_i\leq n+1-i$ for all $i\in[n]$ and  $z_1\leq z_2>z_3\leq z_4>z_5\leq\cdots $.
  It is natural to encode an alternating permutation $\sigma\in \Alt_{n}$ by a sequence $\omega(\sigma)=(z_1, z_2, \ldots, z_n) \in \mathcal{Z}_n$ with $z_i=|\{j\mid \pi_j\leq \pi_i, j\geq i\}|$  for all $i\in[n]$.   It was known~\cite{KPP} that the map $\pi\rightarrow \omega(\pi)$ is a bijection between $\Alt_n$ and $\mathcal{Z}_n$. The following feature of the encoding $\omega$ is clear.

  \begin{lemma}\label{lem-omega}
  	For $n\geq 2$, the map $\omega$ induces a bijection between $\Alt_n$ and $\mathcal{Z}_n$ such that for any $\sigma=\sigma_1\sigma_2\ldots \sigma_n\in \Alt_n$,  we have $\omega(\sigma)=(z_1, z_2, \ldots, z_n)$ with 
  	\begin{equation}\label{eq-omega}
  		z_1=\sigma_1\quad\text{and}\quad z_2=\sigma_2-1.
  	\end{equation}
  \end{lemma}
  In the following, we give a description of a bijection  $\mu$ between $\mathcal{Z}_n$ and $\mathcal{ET}_n$ introduced by  Kuznetsov-Pak-Postnikov \cite{KPP}.  Let \(z=(z_1,z_2,\ldots,z_n)\in\mathcal Z_n\). We construct
  \(\mu(z)\) by recursively inserting nodes into a forest. Start with the
  empty forest \(F_0\). Suppose that \(F_{r-1}\) has already been determined after $r-1$ steps. 
  An unused label \(a\) is called \blue{\em admissible} if the number of roots of
  \(F_{r-1}\) whose labels are larger than \(a\) is even.
  List all admissible labels in increasing order:
  \[
  \mathcal C_{r-1}
  =\{c_1<c_2<\cdots<c_s\}.
  \]
  Define \(p_r=z_r\) if \(r\) is odd, and  \(p_r=n-r+2-z_r\) otherwise. Then set \(i_r=c_{p_r}\). 
  Insert a new vertex labeled \(i_r\) as a root. Make all roots of
  \(F_{r-1}\) whose labels are larger than \(i_r\) the children of \(i_r\),
  arranged from left to right in increasing order. Denote the resulting
  forest by \(F_r\).   
  After the \(n\)th step, adjoin a new root labeled \(0\) and attach all
  roots of \(F_n\) to it, again in increasing order from left to right.
  The resulting tree is defined to be \(\mu(z)\).

  For example, let
  \(
  \sigma=2\,5\,4\,9\,3\,7\,1\,(11)\,8\,(10)\,6
  \in\operatorname{Alt}_{11}.
  \)
  Then
  \[
  z=\omega(\sigma)=(2,4,3,6,2,3,1,4,2,2,1).
  \]
  At the first two steps, we have
  \[
  \begin{aligned}
  	\mathcal C_0&=[11],
  	&p_1&=z_1=2,
  	&i_1&=2,\\
  	\mathcal C_1&=\{3,4,5,6,7,8,9,10,11\},
  	&p_2&=11-2+2-z_2=7,
  	&i_2&=9.
  \end{aligned}
  \]
  For another nontrivial step,
  $\mathcal C_3=\{3,4,5,6,7,8\}$ and $p_4=3$, so that $i_4=5$.
  The two roots $9$ and $11$ are larger than $5$, and hence they become the
  children of $5$.  Continuing in this way gives
  \[
  \begin{aligned}
  	(p_1,p_2,\ldots,p_{11})
  	&=(2,7,3,3,2,4,1,1,2,1,1),\\
  	(i_1,i_2,\ldots,i_{11})
  	&=(2,9,11,5,6,8,1,3,7,10,4).
  \end{aligned}
  \]
  The complete construction is displayed in
  Fig.~\ref{fig:kpp-forward-process}.  
  After adjoining the root $0$, the resulting tree is precisely the tree $T$
  displayed in Fig.~\ref{fig:kpp-forward-process}.  
  
  \begin{figure}[htbp]
  	\centering
  	
  	\begin{tikzpicture}[
  		scale=0.9,
  		transform shape,
  		kpp edge/.style={
  			draw=black,
  			line width=0.45pt,
  			line cap=round,
  			line join=round
  		},
  		kpp new edge/.style={
  			draw=blue!90!black,
  			line width=0.85pt,
  			line cap=round,
  			line join=round
  		},
  		kpp point/.style={
  			circle,
  			inner sep=0pt,
  			outer sep=0pt,
  			minimum size=3.2pt
  		},
  		kpp old point/.style={
  			kpp point,
  			fill=black
  		},
  		kpp new point/.style={
  			kpp point,
  			fill=red
  		},
  		kpp number/.style={
  			font=\scriptsize,
  			text=black,
  			inner sep=0.25pt,
  			outer sep=0pt,
  			anchor=center
  		},
  		kpp stage/.style={
  			font=\footnotesize,
  			inner sep=0pt
  		},
  		kpp tree box/.style={
  			rectangle,
  			draw=black!70,
  			fill=none,
  			line width=0.45pt,
  			rounded corners=1.2pt,
  			inner xsep=2.2mm,
  			inner ysep=1.8mm
  		},
  		transformation arrow/.style={
  			-{Stealth[length=2.2mm,width=1.5mm]},
  			draw=black,
  			line width=0.55pt,
  			line cap=round,
  			line join=round
  		},
  		%
  		f1 tree/.style ={shift={(-7.00, 0.00)}},
  		f2 tree/.style ={shift={(-4.20, 0.00)}},
  		f3 tree/.style ={shift={(-1.20, 0.00)}},
  		f4 tree/.style ={shift={( 2.20, 0.00)}},
  		f5 tree/.style ={shift={( 6.00, 0.00)}},
  		f6 tree/.style ={shift={(-6.30,-2.85)}},
  		f7 tree/.style ={shift={(-2.20,-2.85)}},
  		f8 tree/.style ={shift={( 2.10,-2.85)}},
  		f9 tree/.style ={shift={( 6.20,-2.85)}},
  		f10 tree/.style={shift={(-5.00,-5.70)}},
  		f11 tree/.style={shift={( 0.00,-5.70)}},
  		final tree/.style={shift={( 5.00,-5.70)}}
  		]
  		
  		\newcommand{\KppOldVertex}[5]{%
  			\node[kpp old point] (#1) at (#2,#3) {};
  			\node[kpp number,#5] at (#2,#3) {$#4$};
  		}
  		
  		\newcommand{\KppNewVertex}[5]{%
  			\node[kpp new point] (#1) at (#2,#3) {};
  			\node[kpp number,#5] at (#2,#3) {$#4$};
  		}
  		
  		
  		\begin{scope}[f1 tree,local bounding box=f1contents]
  			\KppNewVertex{f1n2}{0}{0.22}{2}{xshift=0mm,yshift=1.8mm}
  			
  			\node[kpp stage] at (0,-0.83)
  			{$F_1\ (i_1=2)$};
  		\end{scope}
  		
  		\begin{scope}[f2 tree,local bounding box=f2contents]
  			\KppOldVertex{f2n2}{-0.52}{0.22}{2}{xshift=-1.5mm,yshift=1.4mm}
  			\KppNewVertex{f2n9}{ 0.52}{0.22}{9}{xshift=0mm,yshift=1.8mm}
  			
  			\node[kpp stage] at (0,-0.83)
  			{$F_2\ (i_2=9)$};
  		\end{scope}
  		
  		\begin{scope}[f3 tree,local bounding box=f3contents]
  			\KppOldVertex{f3n2}{-0.82}{0.22}{2}{xshift=-1.5mm,yshift=1.4mm}
  			\KppOldVertex{f3n9}{ 0.00}{0.22}{9}{xshift=-1.5mm,yshift=1.4mm}
  			\KppNewVertex{f3n11}{0.82}{0.22}{11}{xshift=0mm,yshift=1.8mm}
  			
  			\node[kpp stage] at (0,-0.83)
  			{$F_3\ (i_3=11)$};
  		\end{scope}
  		
  		\begin{scope}[f4 tree,local bounding box=f4contents]
  			\KppOldVertex{f4n2}{-1.05}{ 0.28}{2}{xshift=-1.5mm,yshift=1.4mm}
  			\KppNewVertex{f4n5}{0.35}{ 0.28}{5}{xshift=0mm,yshift=1.8mm}
  			\KppOldVertex{f4n9}{-0.08}{-0.33}{9}{xshift=-1.5mm,yshift=1.4mm}
  			\KppOldVertex{f4n11}{0.78}{-0.33}{11}{xshift=1.7mm,yshift=1.2mm}
  			
  			\begin{scope}[on background layer]
  				\draw[kpp new edge]
  				(f4n5)--(f4n9)
  				(f4n5)--(f4n11);
  			\end{scope}
  			
  			\node[kpp stage] at (0,-0.83)
  			{$F_4\ (i_4=5)$};
  		\end{scope}
  		
  		\begin{scope}[f5 tree,local bounding box=f5contents]
  			\KppOldVertex{f5n2}{-1.22}{ 0.28}{2}{xshift=-1.5mm,yshift=1.4mm}
  			\KppOldVertex{f5n5}{ 0.00}{ 0.28}{5}{xshift=-1.5mm,yshift=1.4mm}
  			\KppNewVertex{f5n6}{1.22}{ 0.28}{6}{xshift=0mm,yshift=1.8mm}
  			\KppOldVertex{f5n9}{-0.38}{-0.33}{9}{xshift=-1.5mm,yshift=1.4mm}
  			\KppOldVertex{f5n11}{0.38}{-0.33}{11}{xshift=1.7mm,yshift=1.2mm}
  			
  			\begin{scope}[on background layer]
  				\draw[kpp edge]
  				(f5n5)--(f5n9)
  				(f5n5)--(f5n11);
  			\end{scope}
  			
  			\node[kpp stage] at (0,-0.83)
  			{$F_5\ (i_5=6)$};
  		\end{scope}
  		
  		
  		\begin{scope}[f6 tree,local bounding box=f6contents]
  			\KppOldVertex{f6n2}{-1.42}{ 0.28}{2}{xshift=-1.5mm,yshift=1.4mm}
  			\KppOldVertex{f6n5}{-0.47}{ 0.28}{5}{xshift=-1.5mm,yshift=1.4mm}
  			\KppOldVertex{f6n6}{ 0.47}{ 0.28}{6}{xshift=-1.5mm,yshift=1.4mm}
  			\KppNewVertex{f6n8}{1.42}{ 0.28}{8}{xshift=0mm,yshift=1.8mm}
  			\KppOldVertex{f6n9}{-0.80}{-0.33}{9}{xshift=-1.5mm,yshift=1.4mm}
  			\KppOldVertex{f6n11}{-0.14}{-0.33}{11}{xshift=1.7mm,yshift=1.2mm}
  			
  			\begin{scope}[on background layer]
  				\draw[kpp edge]
  				(f6n5)--(f6n9)
  				(f6n5)--(f6n11);
  			\end{scope}
  			
  			\node[kpp stage] at (0,-1.2)
  			{$F_6\ (i_6=8)$};
  		\end{scope}
  		
  		\begin{scope}[f7 tree,local bounding box=f7contents]
  			\KppNewVertex{f7n1}{ 0.00}{ 0.48}{1}{xshift=0mm,yshift=1.8mm}
  			\KppOldVertex{f7n2}{-1.40}{-0.08}{2}{xshift=-1.5mm,yshift=0mm}
  			\KppOldVertex{f7n5}{-0.47}{-0.08}{5}{xshift=-1.5mm,yshift=0mm}
  			\KppOldVertex{f7n6}{ 0.47}{-0.08}{6}{xshift=1.5mm,yshift=0mm}
  			\KppOldVertex{f7n8}{ 1.40}{-0.08}{8}{xshift=1.5mm,yshift=0mm}
  			\KppOldVertex{f7n9}{-0.78}{-0.60}{9}{xshift=-1.5mm,yshift=0mm}
  			\KppOldVertex{f7n11}{-0.16}{-0.60}{11}{xshift=1.7mm,yshift=0mm}
  			
  			\begin{scope}[on background layer]
  				\draw[kpp new edge]
  				(f7n1)--(f7n2)
  				(f7n1)--(f7n5)
  				(f7n1)--(f7n6)
  				(f7n1)--(f7n8);
  				
  				\draw[kpp edge]
  				(f7n5)--(f7n9)
  				(f7n5)--(f7n11);
  			\end{scope}
  			
  			\node[kpp stage] at (0,-1.2)
  			{$F_7\ (i_7=1)$};
  		\end{scope}
  		
  		\begin{scope}[f8 tree,local bounding box=f8contents]
  			\KppOldVertex{f8n1}{-0.53}{ 0.48}{1}{xshift=-1.5mm,yshift=1.4mm}
  			\KppNewVertex{f8n3}{ 1.42}{ 0.48}{3}{xshift=0mm,yshift=1.8mm}
  			\KppOldVertex{f8n2}{-1.52}{-0.08}{2}{xshift=-1.5mm,yshift=0mm}
  			\KppOldVertex{f8n5}{-0.85}{-0.08}{5}{xshift=-1.5mm,yshift=0mm}
  			\KppOldVertex{f8n6}{-0.18}{-0.08}{6}{xshift=1.5mm,yshift=0mm}
  			\KppOldVertex{f8n8}{ 0.49}{-0.08}{8}{xshift=1.5mm,yshift=0mm}
  			\KppOldVertex{f8n9}{-1.12}{-0.60}{9}{xshift=-1.5mm,yshift=0mm}
  			\KppOldVertex{f8n11}{-0.58}{-0.60}{11}{xshift=1.7mm,yshift=0mm}
  			
  			\begin{scope}[on background layer]
  				\draw[kpp edge]
  				(f8n1)--(f8n2)
  				(f8n1)--(f8n5)
  				(f8n1)--(f8n6)
  				(f8n1)--(f8n8)
  				(f8n5)--(f8n9)
  				(f8n5)--(f8n11);
  			\end{scope}
  			
  			\node[kpp stage] at (0,-1.2)
  			{$F_8\ (i_8=3)$};
  		\end{scope}
  		
  		\begin{scope}[f9 tree,local bounding box=f9contents]
  			\KppOldVertex{f9n1}{-0.72}{ 0.48}{1}{xshift=-1.5mm,yshift=1.4mm}
  			\KppOldVertex{f9n3}{ 0.78}{ 0.48}{3}{xshift=-1.5mm,yshift=1.4mm}
  			\KppNewVertex{f9n7}{ 1.48}{ 0.48}{7}{xshift=0mm,yshift=1.8mm}
  			\KppOldVertex{f9n2}{-1.55}{-0.08}{2}{xshift=-1.5mm,yshift=0mm}
  			\KppOldVertex{f9n5}{-1.02}{-0.08}{5}{xshift=-1.5mm,yshift=0mm}
  			\KppOldVertex{f9n6}{-0.49}{-0.08}{6}{xshift=1.5mm,yshift=0mm}
  			\KppOldVertex{f9n8}{ 0.04}{-0.08}{8}{xshift=1.5mm,yshift=0mm}
  			\KppOldVertex{f9n9}{-1.24}{-0.60}{9}{xshift=-1.5mm,yshift=0mm}
  			\KppOldVertex{f9n11}{-0.80}{-0.60}{11}{xshift=1.7mm,yshift=0mm}
  			
  			\begin{scope}[on background layer]
  				\draw[kpp edge]
  				(f9n1)--(f9n2)
  				(f9n1)--(f9n5)
  				(f9n1)--(f9n6)
  				(f9n1)--(f9n8)
  				(f9n5)--(f9n9)
  				(f9n5)--(f9n11);
  			\end{scope}
  			
  			\node[kpp stage] at (0,-1.2)
  			{$F_9\ (i_9=7)$};
  		\end{scope}
  		
  		
  		\begin{scope}[f10 tree,local bounding box=f10contents]
  			\KppOldVertex{f10n1}{-0.85}{ 0.48}{1}{xshift=-1.5mm,yshift=1.4mm}
  			\KppOldVertex{f10n3}{ 0.34}{ 0.48}{3}{xshift=-1.5mm,yshift=1.4mm}
  			\KppOldVertex{f10n7}{ 0.91}{ 0.48}{7}{xshift=-1.5mm,yshift=1.4mm}
  			\KppNewVertex{f10n10}{1.51}{ 0.48}{10}{xshift=0mm,yshift=1.8mm}
  			\KppOldVertex{f10n2}{-1.57}{-0.08}{2}{xshift=-1.5mm,yshift=0mm}
  			\KppOldVertex{f10n5}{-1.08}{-0.08}{5}{xshift=-1.5mm,yshift=0mm}
  			\KppOldVertex{f10n6}{-0.59}{-0.08}{6}{xshift=1.5mm,yshift=0mm}
  			\KppOldVertex{f10n8}{-0.10}{-0.08}{8}{xshift=1.5mm,yshift=0mm}
  			\KppOldVertex{f10n9}{-1.28}{-0.60}{9}{xshift=-1.5mm,yshift=0mm}
  			\KppOldVertex{f10n11}{-0.88}{-0.60}{11}{xshift=1.7mm,yshift=0mm}
  			
  			\begin{scope}[on background layer]
  				\draw[kpp edge]
  				(f10n1)--(f10n2)
  				(f10n1)--(f10n5)
  				(f10n1)--(f10n6)
  				(f10n1)--(f10n8)
  				(f10n5)--(f10n9)
  				(f10n5)--(f10n11);
  			\end{scope}
  			
  			\node[kpp stage] at (0,-1.2)
  			{$F_{10}\ (i_{10}=10)$};
  		\end{scope}
  		
  		\begin{scope}[f11 tree,local bounding box=f11contents]
  			\KppOldVertex{f11n1}{-0.90}{ 0.48}{1}{xshift=-1.5mm,yshift=1.4mm}
  			\KppOldVertex{f11n3}{ 0.23}{ 0.48}{3}{xshift=-1.5mm,yshift=1.4mm}
  			\KppNewVertex{f11n4}{1.12}{ 0.48}{4}{xshift=0mm,yshift=1.8mm}
  			\KppOldVertex{f11n2}{-1.58}{-0.08}{2}{xshift=-1.5mm,yshift=0mm}
  			\KppOldVertex{f11n5}{-1.10}{-0.08}{5}{xshift=-1.5mm,yshift=0mm}
  			\KppOldVertex{f11n6}{-0.62}{-0.08}{6}{xshift=1.5mm,yshift=0mm}
  			\KppOldVertex{f11n8}{-0.14}{-0.08}{8}{xshift=1.5mm,yshift=0mm}
  			\KppOldVertex{f11n7}{ 0.82}{-0.08}{7}{xshift=-1.5mm,yshift=0mm}
  			\KppOldVertex{f11n10}{1.42}{-0.08}{10}{xshift=1.9mm,yshift=0mm}
  			\KppOldVertex{f11n9}{-1.30}{-0.60}{9}{xshift=-1.5mm,yshift=0mm}
  			\KppOldVertex{f11n11}{-0.90}{-0.60}{11}{xshift=1.7mm,yshift=0mm}
  			
  			\begin{scope}[on background layer]
  				\draw[kpp edge]
  				(f11n1)--(f11n2)
  				(f11n1)--(f11n5)
  				(f11n1)--(f11n6)
  				(f11n1)--(f11n8)
  				(f11n5)--(f11n9)
  				(f11n5)--(f11n11);
  				
  				\draw[kpp new edge]
  				(f11n4)--(f11n7)
  				(f11n4)--(f11n10);
  			\end{scope}
  			
  			\node[kpp stage] at (0,-1.2)
  			{$F_{11}\ (i_{11}=4)$};
  		\end{scope}
  		
  		\begin{scope}[final tree,local bounding box=finalcontents]
  			\KppNewVertex{ftn0}{ 0.00}{ 0.64}{0}{xshift=0mm,yshift=1.8mm}
  			\KppOldVertex{ftn1}{-0.95}{ 0.22}{1}{xshift=-1.5mm,yshift=1.4mm}
  			\KppOldVertex{ftn3}{ 0.08}{ 0.22}{3}{xshift=-1.6mm,yshift=0mm}
  			\KppOldVertex{ftn4}{ 1.05}{ 0.22}{4}{xshift=1.5mm,yshift=1.1mm}
  			\KppOldVertex{ftn2}{-1.58}{-0.20}{2}{xshift=-1.5mm,yshift=0mm}
  			\KppOldVertex{ftn5}{-1.14}{-0.20}{5}{xshift=-1.5mm,yshift=0mm}
  			\KppOldVertex{ftn6}{-0.70}{-0.20}{6}{xshift=1.5mm,yshift=0mm}
  			\KppOldVertex{ftn8}{-0.26}{-0.20}{8}{xshift=1.5mm,yshift=0mm}
  			\KppOldVertex{ftn7}{ 0.79}{-0.20}{7}{xshift=-1.5mm,yshift=0mm}
  			\KppOldVertex{ftn10}{1.31}{-0.20}{10}{xshift=1.9mm,yshift=0mm}
  			\KppOldVertex{ftn9}{-1.32}{-0.62}{9}{xshift=-1.5mm,yshift=0mm}
  			\KppOldVertex{ftn11}{-0.96}{-0.62}{11}{xshift=1.7mm,yshift=0mm}
  			
  			\begin{scope}[on background layer]
  				\draw[kpp new edge]
  				(ftn0)--(ftn1)
  				(ftn0)--(ftn3)
  				(ftn0)--(ftn4);
  				
  				\draw[kpp edge]
  				(ftn1)--(ftn2)
  				(ftn1)--(ftn5)
  				(ftn1)--(ftn6)
  				(ftn1)--(ftn8)
  				(ftn5)--(ftn9)
  				(ftn5)--(ftn11)
  				(ftn4)--(ftn7)
  				(ftn4)--(ftn10);
  			\end{scope}
  			
  			\node[kpp stage] at (0,-1.2)
  			{$\mu(z)=T$};
  		\end{scope}
  		
  		%
  		\begin{scope}[on background layer]
  			\node[kpp tree box,fit=(f1contents)] {};
  			\node[kpp tree box,fit=(f2contents)] {};
  			\node[kpp tree box,fit=(f3contents)] {};
  			\node[kpp tree box,fit=(f4contents)] {};
  			\node[kpp tree box,fit=(f5contents)] {};
  			\node[kpp tree box,fit=(f6contents)] {};
  			\node[kpp tree box,fit=(f7contents)] {};
  			\node[kpp tree box,fit=(f8contents)] {};
  			\node[kpp tree box,fit=(f9contents)] {};
  			\node[kpp tree box,fit=(f10contents)] {};
  			\node[kpp tree box,fit=(f11contents)] {};
  			\node[kpp tree box,fit=(finalcontents)] {};
  		\end{scope}
  		
  		%
  		
  		\coordinate (a12s) at (-5.9, 0.05);
  		\coordinate (a12e) at (-5.3, 0.05);
  		
  		\coordinate (a23s) at (-3.10, 0.05);
  		\coordinate (a23e) at (-2.5, 0.05);
  		
  		\coordinate (a34s) at (0.05, 0.05);
  		\coordinate (a34e) at (0.7, 0.05);
  		
  		\coordinate (a45s) at ( 3.6, 0.05);
  		\coordinate (a45e) at ( 4.3, 0.05);

  		\coordinate (row2ins) at (-8.8,-2.85);
  		\coordinate (row2ine) at (-8.2,-2.85);
  		
  		\coordinate (a67s) at (-4.55,-2.85);
  		\coordinate (a67e) at (-4.05,-2.85);
  		
  		\coordinate (a78s) at (-0.3,-2.85);
  		\coordinate (a78e) at ( 0.1,-2.85);
  		
  		\coordinate (a89s) at ( 3.85,-2.85);
  		\coordinate (a89e) at ( 4.2,-2.85);

  		\coordinate (row3ins) at (-7.7,-5.70);
  		\coordinate (row3ine) at (-7.1,-5.70);
  		
  		\coordinate (a1011s) at (-3.1,-5.70);
  		\coordinate (a1011e) at (-2.1,-5.70);
  		
  		\coordinate (a11ts) at (2,-5.70);
  		\coordinate (a11te) at (2.9,-5.70);
  		
  		\draw[transformation arrow] (a12s)--(a12e);
  		\draw[transformation arrow] (a23s)--(a23e);
  		\draw[transformation arrow] (a34s)--(a34e);
  		\draw[transformation arrow] (a45s)--(a45e);

  		\draw[transformation arrow] (row2ins)--(row2ine);
  		\draw[transformation arrow] (a67s)--(a67e);
  		\draw[transformation arrow] (a78s)--(a78e);
  		\draw[transformation arrow] (a89s)--(a89e);

  		\draw[transformation arrow] (row3ins)--(row3ine);
  		\draw[transformation arrow] (a1011s)--(a1011e);
  		\draw[transformation arrow] (a11ts)--(a11te);
  		
  	\end{tikzpicture}
  	
  	\caption{The  construction of  $\mu(z)$, where $z=(2,4,3,6,2,3,1,4,2,2,1)$.}
  	\label{fig:kpp-forward-process}
  \end{figure}
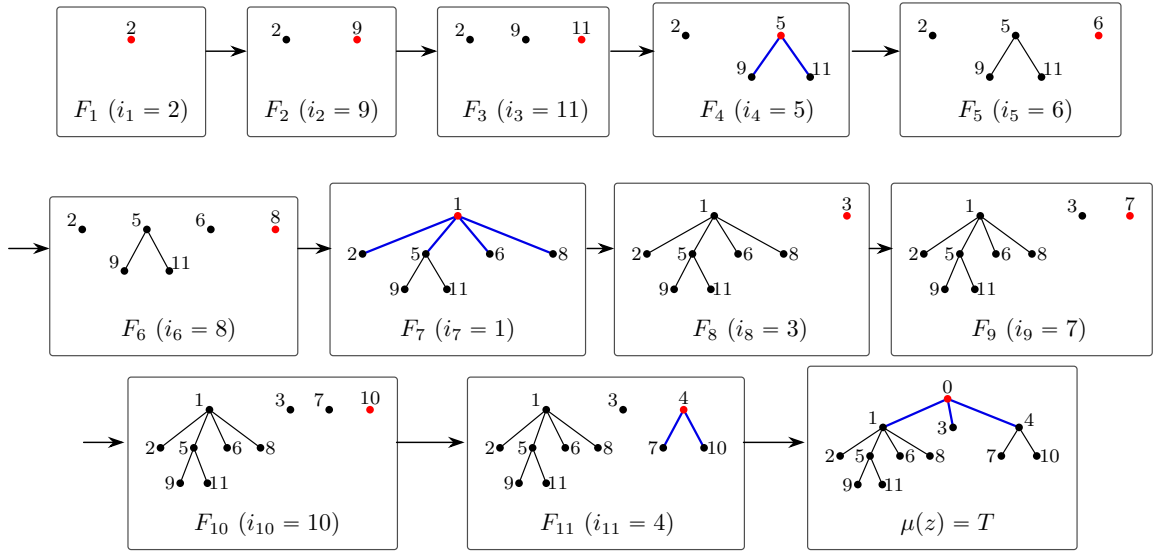

  \begin{lemma}\label{lem-mu}
  	For $n\geq 2$, the map $\mu$ is a bijection between $\mathcal{Z}_n$ and $\mathcal{ET}_n$ such that for any
  	$z=(z_1,z_2,\ldots,z_n)\in\mathcal Z_n$, we have
  	\begin{equation}\label{eq-mu}
  		\fleaf(\mu(z))=z_1\quad\text{and}\quad \sleaf(\mu(z))=n+z_1-z_2.
  	\end{equation}
  \end{lemma}
  
  \begin{proof}
  	Here we use the same notations in the definition of $\mu(z)$. By the construction of $\mu(z)$, the node $i_1$ becomes  the first leaf and the node $i_2$ becomes the second leaf  visited by the preorder traversal  of  the resulting tree $\mu(z)$, that is,  
  	$\fleaf(\mu(z))=i_1$ and $\sleaf(\mu(z))=i_2$.  Recall that  $\mathcal C_0=[n]$ and $\mathcal C_1=[n]\setminus [z_1]$, and hence $i_1=z_1$ and $i_2=z_1+p_2=z_1+n-z_2$, completing the proof.  
  \end{proof}

  The following lemma is a direct consequence of Lemmas~\ref{lem-omega} and~\ref{lem-mu}. 
  \begin{lemma}\label{lem-comp}
  	For $n\geq 2$, the composition  $\nu=\mu\circ\omega$ is a bijection between $\Alt_n$ and $\mathcal{ET}_n$ such that for any  permutation $\sigma=\sigma_1\sigma_2\ldots \sigma_n\in \Alt_n$, we have
  	\begin{equation}\label{eq-comp}
  		\fleaf(\nu(\sigma)))= \sigma_1 \quad \text{and}\quad \sleaf(\nu(\sigma))=n+1+\sigma_1-\sigma_2.
  	\end{equation}
  \end{lemma}
  
  For $n\geq 2$, 
  let $\Alt_{n, i,k}$ denote the set of alternating permutations $\sigma=\sigma_1\sigma_2\ldots \sigma_n\in \Alt_n$ with $\sigma_1=k$ and $\sigma_2=k+i$.
  
  \begin{lemma}\label{lem-Alt}
  	For $n\geq 2$ and $i, k\geq 1$,  we have $|\Alt_{n,i,k}|=E_{n-1, n+1-i-k}$. 
  \end{lemma}
  \begin{proof}
  	Given a permutation $\sigma=\sigma_1\sigma_2\ldots \sigma_n\in \Alt_{n,i,k}$, if $i=1$,  then we have $\sigma_1=k$ and $\sigma_{2}=k+1$.   Let $\tau=\tau_1\tau_2\ldots \tau_{n-1}\in \mathfrak{S}_{n-1}$ be the permutation obtained from $\sigma$ by removing $\sigma_{1}=k$  and decreasing each letter larger than $k$ by one.   Let $\eta(\sigma)=(n+1-\tau_1)(n+1-\tau_2)\ldots (n+1-\tau_{n-1})$. One can easily check that the resulting permutation $\eta(\sigma)$ is an alternating permutation in $\Alt_{n-1}$ whose first letter is $n-k$. Clearly, the map $\sigma\rightarrow \eta(\sigma)$ is a bijection and hence  we have $|\Alt_{n,1,k}|=E_{n-1, n-k}$.

  	Now we assume that $i>1$.
  	Given a permutation $\sigma\in \Alt_{n,i,k}$,  we can get an alternating permutation $\tau=\gamma(\sigma)\in \Alt_{n,i-1,k+1}$ from $\sigma$ by swapping the letters $k$ and $k+1$ in $\sigma$.   It is apparent that the map $\tau\rightarrow \gamma(\sigma)$ is a bijection between $\Alt_{n, i,k}$ and $\Alt_{n,i-1, k+1}$ for all $i>1$. This gives that
  	$$
	|\Alt_{n,i,k}|=|\Alt_{n, i-1, k+1}|=\cdots=|\Alt_{n, 1, i+k-1}|=E_{n-1, n+1-i-k},
	$$ 
	 completing the proof. 
  	 \end{proof}
  
  Now we are in the position to complete the proof of Conjecture~\ref{con2}. 
  
  \begin{proof}[{\bf Proof of Conjecture~\ref{con2}}]	Let $\mathcal{J}_{n,i,k}$ denote the set of Jacobi Permutations $\pi\in \mathcal{J}_n$ with $\last(\pi)=k$ and $\cpen(\pi)=i$. 
  Combining Lemmas~\ref{lem-kappa} and~\ref{lem-comp}, we deduce that
  the map $\delta=\kappa^{-1}\circ \nu$ serves as  a bijection between $\Alt_n$ and $\mathcal{J}_n$ such that for any  permutation $\sigma=\sigma_1\sigma_2\ldots \sigma_n\in \Alt_n$, we have
  $ 
  \last(\delta(\sigma))= \sigma_1 \,\, \mbox{and}\,\,  \cpen(\delta(\sigma))= \sigma_2-\sigma_1.
  $ 
  This means that the map $\delta$ is  a bijection between $\Alt_{n,i,k}$ and $\mathcal{J}_{n,i,k}$. This combined with Lemma~\ref{lem-Alt} gives that
  $j^{(\cpen, \last)}_{n,i,k}=|\mathcal{J}_{n,i,k}|=|\Alt_{n,i,k}|=E_{n-1, n+1-i-k}$ as desired, completing the proof. \end{proof}

  \section{The construction of $\Omega$ and proof of Theorem~\ref{thm-main}}
  
  \label{Sec:3}
  
  This section is devoted to the   construction of the desired  bijection $\Omega$. 
Recall that  a  \blue{\em binary tree} is a rooted tree where each internal node has either a left child, a right child, or both.

 \begin{definition}[Increasing binary trees]
 	Fix a  set $S$ of positive integers.   An  \blue{  increasing binary  tree} on $S$ is a labeled  binary tree   such that
 	\begin{itemize} 
 		\item  the labels of the nodes form precisely the set $S$ and
 		\item  the labels along a path from the root to any leaf is  increasing.
 	\end{itemize}
 	Denote by $\B_S$ the set of all   increasing binary trees on $S$. 
 \end{definition}
 Given a labeled tree $T$ and a node $x$ of $T$,  let $T_{x}$ denote the subtree rooted at the node $x$.  
 Given  a binary tree $T\in \B_S$ and a node $x$ of $T$,  the subtree $T_{y}$ is called a \blue{\em left} (resp., \blue{\em right})  \blue{\em subtree} of the node $x$  if the node $y$ is a left (resp., right) child of the node $x$. 
 Given a binary tree $T\in \B_S$ and an internal node $x$ of $T$,  the node    $x$ is said  to be  an  \blue{\em Andr\'e  $\mathrm{I}$ node}  of $T$ if the maximum  label of the subtree $T_x$  belongs to its right subtree.  A leaf  is always regarded as an  Andr\'e $\mathrm{I}$ node.  
 \begin{definition}[Andr\'e $\mathrm{I}$ trees]
 	Fix a  set $S$ of positive integers.   An  \blue{  Andr\'e $\mathrm{I}$ tree} on $S$ is an increasing binary tree on $S$   such that  all the nodes of $T$ are Andr\'e $\mathrm{I}$ nodes.
 	Denote by $\AB^{\mathrm{I}}_S$ the set of all   Andr\'e $\mathrm{I}$ trees on $S$.    For example,  Fig.~\ref{fig:increasing-binary-tree-19-reference-proportion}  (in right) depicts  an Andr\'e  $\mathrm{I}$ tree $T'$.
 \end{definition}

 We adopt some tree terminology from~\cite{Lin2,Lin3}. Let $T\in \B_S$ be a binary tree.  A \blue{\em right chain} (resp.,~\blue{\em left chain}) in $T$ is a maximal path consisting of only right (resp.,~left) edges. Here, a \blue{\em chain} refers to either a right chain or a left chain.
 The \blue{\em size} of a chain is  its  number of nodes and the \blue{\em parity} of a chain depends on its  size.     
  The right chain (resp., left chain) containing the root of $T$ is called  the \blue{\em right arm}  (resp., \blue{\em left arm}) of $T$.   For example,  if we let $T$ be the increasing binary tree displayed in Fig.~\ref{fig:increasing-binary-tree-19-reference-proportion} (in left), then the left arm of $T$ is given by $1-2-4-5-12$ and the right arm of $T$ is given by $1-7-8$.    Let $\larm(T)$ and  $\orc^*(T)$ denote the size of the left arm of $T$ and the total number of odd right arms excluding the right arm, respectively.  For instance, if we let $T$ be the increasing binary tree displayed in Fig.~\ref{fig:increasing-binary-tree-19-reference-proportion} (in right),  then we have $\larm(T)=3$ and $\orc^*(T)=6$. 
 Given a right chain $v_1-v_2-\cdots-v_k$ of $T$, if $u$ is the parent of $v_1$, then we say that each $v_i$ is a {\em right grandson}  of $u$ for all $1\leq i\leq k$. 
 
 A node $u$ of $T$ is said to be a  \blue{\em Jacobi node} if $u$ has an even number of right grandsons.  Take the increasing  binary tree $T$ displayed in Fig.~\ref{fig:increasing-binary-tree-19-reference-proportion} (in left) for an example.  Clearly,   the node $4$ of $T$  has four grandsons  which are  nodes $5,13,17,18$.  
  
  \begin{definition}[Jacobi trees]
 	Fix a  set $S$ of positive integers.   A \blue{Jacobi tree}  on $S$ is an increasing binary tree on $S$  if each node   is a Jacobi node.   Equivalently,  a Jacobi tree  on $S$ is an increasing binary tree on $S$   with $\orc^{*}(T)=0$.
 	Denote by $\JB_S$ the set of all   increasing binary trees on $S$.  For instance, a Jacobi tree $T$ is displayed in Fig.~\ref{fig:increasing-binary-tree-19-reference-proportion}  (in left). 
 \end{definition}

  For an increasing binary tree $T\in \B_n$ and 
  a left edge $e=uv$  of $T$ where $u$ is the parent of $v$, then node $u$ is called the \blue{\em left top} of $T$.   Let $\Lefttop(T)$  of the set  of all the left tops   of $T$. 
Denote   $\first(T)$  the label of the leftmost  node of the left arm.  Similarly, let $\max(T)$ denote the maximum label of the nodes of $T$.  
  For example, let $T$ be the  tree displayed in Fig.~\ref{fig:increasing-binary-tree-19-reference-proportion} (in left).  Clearly, we have
  $\Lefttop(T)=\{1,2,4,5,6,10, 13\}$,    $\first(T)=12$ and $\Left(T)=7$. 
    
 \begin{figure}
 	\centering
 	
 	\begin{tikzpicture}[
 		scale=1,
 		transform shape,
 		x=0.48cm,
 		y=0.52cm,
 		tree vertex/.style={
 			circle,
 			fill=black,
 			draw=black,
 			inner sep=0pt,
 			minimum size=3.5pt
 		},
 		every label/.style={
 			font=\scriptsize,
 			text=black,
 			inner sep=0.7pt
 		},
 		normal edge/.style={
 			draw=black,
 			line width=0.55pt,
 			line cap=round
 		},
 		transformation arrow/.style={
 			->,
 			draw=black,
 			line width=0.55pt,
 			>=stealth
 		},
 		tree name/.style={
 			font=\normalsize
 		},
 		%
 		T tree/.style={
 			shift={(0.00,0.00)}
 		},
 		Phi T tree/.style={
 			shift={(10.00,0.00)}
 		}
 		]
 		
 		\begin{scope}[T tree]
 			\coordinate (a1)  at ( 0.00, 0.00);
 			\coordinate (a2)  at (-1.00,-1.00);
 			\coordinate (a7)  at ( 1.00,-1.00);
 			\coordinate (a4)  at (-2.00,-2.00);
 			\coordinate (a3)  at ( 0.00,-2.00);
 			\coordinate (a8)  at ( 2.00,-2.00);
 			\coordinate (a5)  at (-3.80,-3.80);
 			\coordinate (a6)  at ( 0.00,-4.00);
 			\coordinate (a12) at (-5.20,-5.20);
 			\coordinate (a13) at (-2.70,-4.90);
 			\coordinate (a9)  at (-0.80,-4.80);
 			\coordinate (a16) at (-4.30,-6.10);
 			\coordinate (a15) at (-3.30,-5.50);
 			\coordinate (a17) at (-2.10,-5.50);
 			\coordinate (a10) at ( 0.00,-5.60);
 			\coordinate (a19) at (-2.70,-6.10);
 			\coordinate (a18) at (-1.50,-6.10);
 			\coordinate (a11) at (-0.80,-6.40);
 			\coordinate (a14) at ( 0.00,-7.20);
 			
 			\draw[normal edge]
 			(a1)--(a2)
 			(a1)--(a7)
 			(a2)--(a4)
 			(a2)--(a3)
 			(a7)--(a8)
 			(a4)--(a5)
 			(a4)--(a6)
 			(a5)--(a12)
 			(a5)--(a13)
 			(a12)--(a16)
 			(a13)--(a15)
 			(a13)--(a17)
 			(a15)--(a19)
 			(a17)--(a18)
 			(a6)--(a9)
 			(a9)--(a10)
 			(a10)--(a11)
 			(a11)--(a14);
 			
 			\node[tree vertex,label=above:{1}]        at (a1)  {};
 			\node[tree vertex,label=above left:{2}]   at (a2)  {};
 			\node[tree vertex,label=above right:{7}]  at (a7)  {};
 			\node[tree vertex,label=above left:{4}]   at (a4)  {};
 			\node[tree vertex,label=above right:{3}]  at (a3)  {};
 			\node[tree vertex,label=above right:{8}]  at (a8)  {};
 			\node[tree vertex,label=above left:{5}]   at (a5)  {};
 			\node[tree vertex,label=above right:{6}]  at (a6)  {};
 			\node[tree vertex,label=left:{12}]        at (a12) {};
 			\node[tree vertex,label=above right:{13}] at (a13) {};
 			\node[tree vertex,label=above left:{9}]   at (a9)  {};
 			\node[tree vertex,label=below:{16}]       at (a16) {};
 			\node[tree vertex,label=left:{15}]        at (a15) {};
 			\node[tree vertex,label=above right:{17}] at (a17) {};
 			\node[tree vertex,label=above right:{10}] at (a10) {};
 			\node[tree vertex,label=below left:{19}]  at (a19) {};
 			\node[tree vertex,label=below left:{18}]  at (a18) {};
 			\node[tree vertex,label=right:{11}]       at (a11) {};
 			\node[tree vertex,label=below right:{14}] at (a14) {};
 		\end{scope}
 		
 		\begin{scope}[Phi T tree]
 			\coordinate (f1)  at ( 0.00, 0.00);
 			\coordinate (f5)  at (-3.00,-3.00);
 			\coordinate (f2)  at ( 1.00,-1.00);
 			\coordinate (f3)  at ( 0.00,-2.00);
 			\coordinate (f4)  at ( 2.00,-2.00);
 			\coordinate (f6)  at ( 1.00,-3.00);
 			\coordinate (f7)  at ( 3.00,-3.00);
 			\coordinate (f8)  at ( 4.00,-4.00);
 			\coordinate (f12) at (-4.40,-4.40);
 			\coordinate (f17) at (-2.40,-3.60);
 			\coordinate (f16) at (-3.50,-5.30);
 			\coordinate (f18) at (-1.80,-4.20);
 			\coordinate (f9)  at (-0.40,-4.40);
 			\coordinate (f10) at ( 2.00,-4.00);
 			\coordinate (f11) at ( 1.00,-5.00);
 			\coordinate (f14) at ( 3.00,-5.00);
 			\coordinate (f13) at ( 5.00,-5.00);
 			\coordinate (f15) at ( 4.00,-6.00);
 			\coordinate (f19) at ( 6.00,-6.00);
 			
 			\draw[normal edge]
 			(f1)--(f5)
 			(f1)--(f2)
 			(f5)--(f12)
 			(f5)--(f17)
 			(f12)--(f16)
 			(f17)--(f18)
 			(f2)--(f3)
 			(f2)--(f4)
 			(f4)--(f6)
 			(f4)--(f7)
 			(f6)--(f9)
 			(f6)--(f10)
 			(f10)--(f11)
 			(f10)--(f14)
 			(f7)--(f8)
 			(f8)--(f13)
 			(f13)--(f15)
 			(f13)--(f19);
 			
 			\node[tree vertex,label=above:{1}]        at (f1)  {};
 			\node[tree vertex,label=above left:{5}]   at (f5)  {};
 			\node[tree vertex,label=above right:{2}]  at (f2)  {};
 			\node[tree vertex,label=above left:{3}]   at (f3)  {};
 			\node[tree vertex,label=above right:{4}]  at (f4)  {};
 			\node[tree vertex,label=above left:{6}]   at (f6)  {};
 			\node[tree vertex,label=above right:{7}]  at (f7)  {};
 			\node[tree vertex,label=above right:{8}]  at (f8)  {};
 			\node[tree vertex,label=left:{12}]        at (f12) {};
 			\node[tree vertex,label=above right:{17}] at (f17) {};
 			\node[tree vertex,label=below right:{16}] at (f16) {};
 			\node[tree vertex,label=below left:{18}]  at (f18) {};
 			\node[tree vertex,label=below left:{9}]   at (f9)  {};
 			\node[tree vertex,label=above right:{10}] at (f10) {};
 			\node[tree vertex,label=below left:{11}]  at (f11) {};
 			\node[tree vertex,label=below right:{14}] at (f14) {};
 			\node[tree vertex,label=above right:{13}] at (f13) {};
 			\node[tree vertex,label=below left:{15}]  at (f15) {};
 			\node[tree vertex,label=below right:{19}] at (f19) {};
 		\end{scope}
 		
 		

 	\end{tikzpicture}
 	
 	\caption{A Jacobi tree (in left) and an Andr\'e $\mathrm{I}$ tree (in right).}
 	\label{fig:increasing-binary-tree-19-reference-proportion}
 \end{figure}
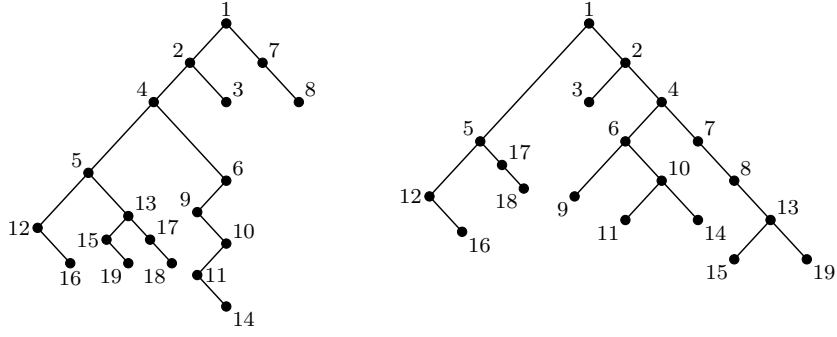

     \begin{lemma}\label{thm-xi}
   	Fix the set $S$ of positive integers.
   	There is a one-to-one correspondence $\xi$  between     $ \mathcal{J}_{S}$ and $\JB_S$ such that   for any $\pi\in \mathcal{J}_S$, we have 
   	\begin{equation}\label{eq-xi}
   	(\Ascbot, \last)\pi=(\Lefttop, \first)\xi(\pi).
   	\end{equation}
   \end{lemma}
   
   \begin{proof}
   	Given a Jacobi permutation $\pi\in \mathcal{J}_S$, if   $\pi$ is empty, set $\xi(\pi)$ to be the empty tree.  Otherwise,  $\pi$ can be uniquely decomposed as $u\min(\pi) v$, where $u,v$ are (possibly empty) Jacobi permutations and $|v|$ is even.  Set the node $\min(\pi)$ to be the root of $\xi(\pi)$. Then attach the tree $\xi(v)$ as the left subtree of the node $\min(\pi)$ and attach the tree $\xi(u)$ as  the right subtree of the node $\min(\pi)$.  For example,  if we let $\pi=871326(10)(14)(11)94(18)(17)(13)(19)(15)5(16)(12)
   	\in \mathcal{J}_{19}$, then its corresponding tree $T=\xi(\pi)$ is illustated in Fig.~\ref{fig:increasing-binary-tree-19-reference-proportion} (in Left).

   	Now we proceed to show that the resulting tree $\xi(\pi)$ is a Jacobi tree satisfying~\eqref{eq-xi} by induction on $S$.  Clearly the assertion holds for $|S|=1$. Assume that $\xi(\sigma)$ is a Jacobi tree with $
   	(\Ascbot, \last)\sigma=(\Lefttop, \first)\xi(\sigma) $ for any $\sigma\in \mathcal{J}_{S'}$ with $|S'|<|S|$. 
   	
   	By the induction hypothesis,  the resulting trees $\xi(u)$ and $\xi(v)$ are  Jacobi trees. Since $|v|$ is even,  every right chain of  $\xi(v)$  has an even size.   Then by the construction of $\xi(T)$, it is easily seen that $\xi(T)\in \JB_S$.  
   	
   	Now we proceed to show that $ 
   	(\Ascbot, \last)\pi=(\Lefttop, \first)\xi(\pi)
   	$. We have two cases.
   	\begin{itemize} 
   	
   	\item {\bf Case 1: $|v|=0$.} In this case,  we have $\pi=u\min(\pi)$ and the root of $\xi(\pi)$ has no left child.  By the construction of $\xi(T)$ and the induction hypothesis, we deduce that
   $\last(\pi)=\min(\pi)=\first(\xi(\pi))$ and $\Lefttop(\xi(\pi))=\Lefttop(\xi(u))=\Ascbot(u)=\Ascbot(\pi)$, as desired.
  \item{\bf Case 2: $|v|>0$.} 
    By the construction of $\xi(T)$ and the induction hypothesis, we deduce that
   $\last(\pi)=\last(v)=\first(\xi(v))=\first(\xi(\pi))$ and 
   \[
   \begin{aligned}
   \Lefttop(\xi(\pi))&=\Lefttop(\xi(u))\cup\Lefttop(\xi(v))\cup\{\min(\pi)\}\\
   &=\Ascbot(u)\cup \Ascbot(v)\cup\{\min(\pi)\}\\
   &=\Ascbot(\pi),
   \end{aligned}
   \] as desired.
   \end{itemize}
    This completes the proof of~\eqref{eq-xi} by induction. It is apparent that the construction of $\xi$ is reversible and hence  the map 
    $\xi$ is the desired   bijection, completing the proof.
   \end{proof}

     It is well known that there exists a bijection $\theta: \mathfrak{S}_S\rightarrow \B_S$ (see~\cite[Chapter 1]{ST}), called tree representation of permutations. Given a permutation $\pi\in \mathfrak{S}_S$, if   $\pi$ is empty, set $\theta(\pi)$ to be the empty tree.   Otherwise,  $\pi$ can be uniquely decomposed as $u\min(\pi) v$, where $u,v$ are (possibly empty)   permutations. Then attach the tree $\theta(u)$ as the left subtree of the node $\min(\pi)$ and attach the tree $\theta(v)$ as  the right subtree of the node $\min(\pi)$.  For example,  if we let $\pi=  (12)(16)5(17)(18)13296(11)(10)(14)478(15)(13)(19)
    $, then its corresponding tree $\theta(\pi)$ is illustated in Fig.~\ref{fig:increasing-binary-tree-19-reference-proportion} (in right). 
   The following property follows directly from the definitions of the map $\theta$ and Andr\'e $\mathrm{I}$ permutations.
   
    \begin{lemma}\label{thm-theta}
   	Fix the set $S$ of positive integers. The map $\theta$ induces  a one-to-one correspondence    between     $ \mathcal{A}^{\mathrm{I}}_{S}$ and $\AB^{\mathrm{I}}_S$ such that   for any $\pi\in \mathcal{A}^{\mathrm{I}}_S$, we have 
   	\begin{equation}\label{eq-theta}
   		(\Desbot, \first)\pi=(\Lefttop, \first)\theta(\pi).
   	\end{equation}
   \end{lemma}

   In views of Lemmas~\ref{thm-xi} and~\ref{thm-theta},  to construct the bijection $\Omega$, it remains to establsih the bijection between Jacobi trees and Andr\'e $\mathrm{I}$ trees. 
    \begin{theorem}\label{thm-Phi}
    	Fix  the set $S$ of positive integers. There exists a bijection $\Phi: \JB_S \rightarrow \AB^{\mathrm{I}}_S$ such that for any $T\in \JB_S$, we have 
    	\begin{equation}\label{eq-Phi}
    		(\Lefttop,   \first)T=(\Lefttop,   \first))\Phi(T).
    		\end{equation}
    	\end{theorem}

 Before we establish the bijection $\Phi$, we shall introduce a curcial transformation $\phi$ on Jacobi trees.
     
    {\bf The transformation $\phi$.} Given a Jacobi tree $T\in \JB_S$ rooted at the node $x$, if $x$ is an     Andr\'e $\mathrm{I}$ node, set $\phi(T)=T$. Otherwise,  find the topmost node $y$ on the left arm of $T$  such that the node $\max(T)$  belongs to the right subtree of  $y$.      
    Find the node $z_0$ on the right chain starting with the node $y$ such that either $\max(T)=z_0$ or  the node $\max(T)$ belongs to the left subtree of the node $z_0$.  Set $T^{(0)}=T$. 
      Then we can generate an increasing binary tree $T^{(1)}=\alpha(T^{(0)}, z_0)$  
    	 by the following procedure.
    	     \begin{itemize}
    	     	\item  Let $L_1$ be the subtree obtained from $T_{z_0}$ by removing the right subtree of the node $z_0$.  Remove  the subtree $L_1$ from $T$. 
    	     	\item Attach the subtree $L_1$  to the resulting tree  by inserting the node $z_0$ into the right chain starting with the node $z_1$ so that the labels along this chain are increasing from left to right, where $z_1$ is the parent of the node $y$.  
    	     	\end{itemize}
    	If $z_1=x$, set $\phi(T)=T^{(1)}$.  Otherwise,   we can generate a pair $\beta(T^{(1)}, z_1)=(T^{(2)}, z_2)$ from $T^{(1)}$ as follows.
    	\begin{itemize}
    		\item  Let $L_2$ be the right  subtree  of   the node $z_1$ in $T^{(1)}$.
    		 Delete  the edges $z_2z_1$ and $z_1y$,  and connect the nodes $z_2$ and $y$ by a left edge, where $z_2$ is the parent of $z_1$ in $T^{(1)}$. 
  \item  Insert the node $z_1$ into the right chain starting with the node $z_2$ so that the labels are increasing from left to right along this chain.  
    		\item  Attach the subtree  $L_2$  to the resulting tree by connecting its root to the node $z_1$ by a left edge.
    		\end{itemize}
    	If $z_2=x$, we stop and set $\phi(T)=T^{(2)}$. Otherwise, we repeat the transformation $\beta$ for $(T^{(2)}, z_2)$. We continue this process until we reach a pair  $(T^{(t)}, z_t)$ so that  $x=z_t$. Then set $\phi(T)=T^{(t)}$.  See Fig.~\ref{fig:phi} for an illustration, where $y=5$, $z_0=13$, $z_1=4$, $z_2=2$, $z_3=1$ and  the subtrees $L_1$, $L_2$ and $L_3$ are colored   red.

    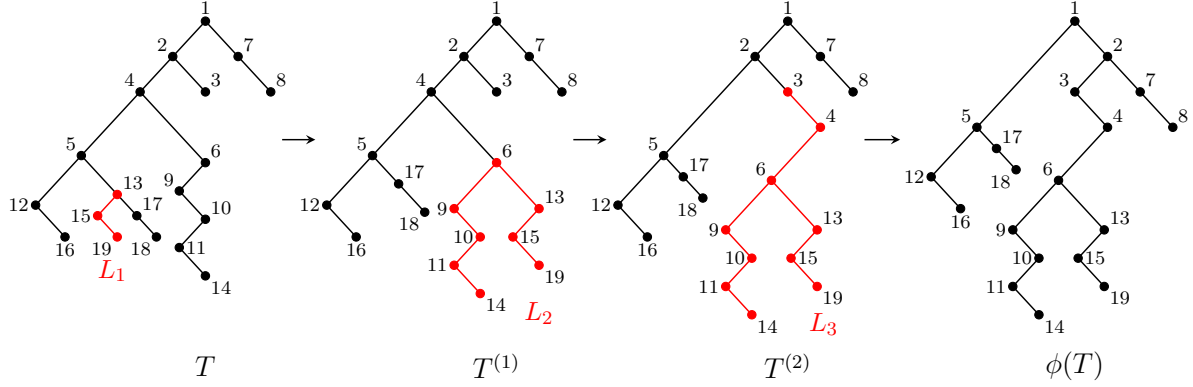
\begin{figure}[htbp]
    	\centering
    	
    	\begin{tikzpicture}[
    		scale=0.9,
    		transform shape,
    		x=0.48cm,
    		y=0.52cm,
    		tree vertex/.style={
    			circle,
    			fill=black,
    			draw=black,
    			inner sep=0pt,
    			minimum size=3.5pt
    		},
    		highlighted vertex/.style={
    			circle,
    			fill=red,
    			draw=red,
    			inner sep=0pt,
    			minimum size=3.5pt
    		},
    		every label/.style={
    			font=\scriptsize,
    			text=black,
    			inner sep=0.7pt
    		},
    		normal edge/.style={
    			draw=black,
    			line width=0.55pt,
    			line cap=round
    		},
    		highlighted edge/.style={
    			draw=red,
    			line width=0.55pt,
    			line cap=round
    		},
    		transformation arrow/.style={
    			->,
    			draw=black,
    			line width=0.55pt,
    			>=stealth
    		},
    		tree name/.style={
    			font=\normalsize
    		},
    		highlighted name/.style={
    			font=\small,
    			text=red
    		},
    		second tree/.style={
    			xshift=4.28cm
    		},
    		third tree/.style={
    			xshift=8.56cm
    		},
    		fourth tree/.style={
    			xshift=12.78cm
    		}
    		]
    		\begin{scope}
    			\coordinate (a1)  at ( 0.00, 0.00);
    			\coordinate (a2)  at (-1.00,-1.00);
    			\coordinate (a7)  at ( 1.00,-1.00);
    			\coordinate (a4)  at (-2.00,-2.00);
    			\coordinate (a3)  at ( 0.00,-2.00);
    			\coordinate (a8)  at ( 2.00,-2.00);
    			\coordinate (a5)  at (-3.80,-3.80);
    			\coordinate (a6)  at ( 0.00,-4.00);
    			\coordinate (a12) at (-5.20,-5.20);
    			\coordinate (a13) at (-2.70,-4.90);
    			\coordinate (a9)  at (-0.80,-4.80);
    			\coordinate (a16) at (-4.30,-6.10);
    			\coordinate (a15) at (-3.30,-5.50);
    			\coordinate (a17) at (-2.10,-5.50);
    			\coordinate (a10) at ( 0.00,-5.60);
    			\coordinate (a19) at (-2.70,-6.10);
    			\coordinate (a18) at (-1.50,-6.10);
    			\coordinate (a11) at (-0.80,-6.40);
    			\coordinate (a14) at ( 0.00,-7.20);
    			
    			\draw[normal edge]
    			(a1)--(a2)
    			(a1)--(a7)
    			(a2)--(a4)
    			(a2)--(a3)
    			(a7)--(a8)
    			(a4)--(a5)
    			(a4)--(a6)
    			(a5)--(a12)
    			(a5)--(a13)
    			(a12)--(a16)
    			(a13)--(a17)
    			(a17)--(a18)
    			(a6)--(a9)
    			(a9)--(a10)
    			(a10)--(a11)
    			(a11)--(a14);
    			
    			\draw[highlighted edge]
    			(a13)--(a15)
    			(a15)--(a19);
    			
    			\node[tree vertex,label=above:{1}]              at (a1)  {};
    			\node[tree vertex,label=above left:{2}]         at (a2)  {};
    			\node[tree vertex,label=above right:{7}]        at (a7)  {};
    			\node[tree vertex,label=above left:{4}]               at (a4)  {};
    			\node[tree vertex,label=above right:{3}]        at (a3)  {};
    			\node[tree vertex,label=above right:{8}]        at (a8)  {};
    			\node[tree vertex,label=above left:{5}]               at (a5)  {};
    			\node[tree vertex,label=above right:{6}]        at (a6)  {};
    			\node[tree vertex,label=left:{12}]              at (a12) {};
    			\node[highlighted vertex,label=above right:{13}] at (a13) {};
    			\node[tree vertex,label=above left:{9}]         at (a9)  {};
    			\node[tree vertex,label=below:{16}]             at (a16) {};
    			\node[highlighted vertex,label=left:{15}]       at (a15) {};
    			\node[tree vertex,label=above right:{17}]        at (a17) {};
    			\node[tree vertex,label=above right:{10}]       at (a10) {};
    			\node[highlighted vertex,label=below left:{19}] at (a19) {};
    			\node[tree vertex,label=below left:{18}]       at (a18) {};
    			\node[tree vertex,label=right:{11}]        at (a11) {};
    			\node[tree vertex,label=below right:{14}]       at (a14) {};
    			
    			\node[highlighted name] at (-2.85,-7.05) {\(L_1\)};
    			\node[tree name] at (0,-9.75) {\(T\)};
    			
    			\coordinate (first-arrow-start) at ( 2.35,-3.35);
    		\end{scope}
    		
    		\begin{scope}[second tree]
    			\coordinate (b1)  at ( 0.00, 0.00);
    			\coordinate (b2)  at (-1.00,-1.00);
    			\coordinate (b7)  at ( 1.00,-1.00);
    			\coordinate (b4)  at (-2.00,-2.00);
    			\coordinate (b3)  at ( 0.00,-2.00);
    			\coordinate (b8)  at ( 2.00,-2.00);
    			\coordinate (b5)  at (-3.80,-3.80);
    			\coordinate (b6)  at ( 0.00,-4.00);
    			\coordinate (b12) at (-5.20,-5.20);
    			\coordinate (b17) at (-3.00,-4.60);
    			\coordinate (b9)  at (-1.30,-5.30);
    			\coordinate (b13) at ( 1.30,-5.30);
    			\coordinate (b16) at (-4.30,-6.10);
    			\coordinate (b18) at (-2.20,-5.40);
    			\coordinate (b10) at (-0.50,-6.10);
    			\coordinate (b15) at ( 0.50,-6.10);
    			\coordinate (b11) at (-1.30,-6.90);
    			\coordinate (b19) at ( 1.30,-6.90);
    			\coordinate (b14) at (-0.50,-7.70);
    			
    			\draw[normal edge]
    			(b1)--(b2)
    			(b1)--(b7)
    			(b2)--(b4)
    			(b4)--(b6)
    			(b2)--(b3)
    			(b7)--(b8)
    			(b4)--(b5)
    			(b5)--(b12)
    			(b5)--(b17)
    			(b12)--(b16)
    			(b17)--(b18);
    			
    			\draw[highlighted edge]
    			
    			(b6)--(b9)
    			(b6)--(b13)
    			(b9)--(b10)
    			(b10)--(b11)
    			(b11)--(b14)
    			(b13)--(b15)
    			(b15)--(b19);
    			
    			\node[tree vertex,label=above:{1}]              at (b1)  {};
    			\node[tree vertex,label=above left:{2}]         at (b2)  {};
    			\node[tree vertex,label=above right:{7}]        at (b7)  {};
    			\node[tree vertex,label=above left:{4}]        at (b4)  {};
    			\node[tree vertex,label=above right:{3}]        at (b3)  {};
    			\node[tree vertex,label=above right:{8}]        at (b8)  {};
    			\node[tree vertex,label=above left:{5}]               at (b5)  {};
    			\node[highlighted vertex,label=above right:{6}]  at (b6)  {};
    			\node[tree vertex,label=left:{12}]              at (b12) {};
    			\node[tree vertex,label=above right:{17}]       at (b17) {};
    			\node[highlighted vertex,label=left:{9}]        at (b9)  {};
    			\node[highlighted vertex,label=above right:{13}] at (b13) {};
    			\node[tree vertex,label=below:{16}]             at (b16) {};
    			\node[tree vertex,label=below left:{18}]        at (b18) {};
    			\node[highlighted vertex,label=left:{10}] at (b10) {};
    			\node[highlighted vertex,label=right:{15}] at (b15) {};
    			\node[highlighted vertex,label=left:{11}]       at (b11) {};
    			\node[highlighted vertex,label=below right:{19}] at (b19) {};
    			\node[highlighted vertex,label=below right:{14}] at (b14) {};
    			
    			\node[highlighted name,anchor=west] at (0.55,-8.25) {\(L_2\)};
    			\node[tree name] at (0,-9.75) {\(T^{(1)}\)};
    			
    			\coordinate (first-arrow-end)    at (-5.55,-3.35);
    			\coordinate (second-arrow-start) at ( 2.35,-3.35);
    		\end{scope}
    		
    		\begin{scope}[third tree]
    			\coordinate (c1)  at ( 0.00, 0.00);
    			\coordinate (c2)  at (-1.00,-1.00);
    			\coordinate (c7)  at ( 1.00,-1.00);
    			\coordinate (c3)  at ( 0.00,-2.00);
    			\coordinate (c8)  at ( 2.00,-2.00);
    			\coordinate (c4)  at ( 1.00,-3.00);
    			\coordinate (c5)  at (-3.80,-3.80);
    			\coordinate (c6)  at (-0.50,-4.50);
    			\coordinate (c12) at (-5.20,-5.20);
    			\coordinate (c17) at (-3.20,-4.40);
    			\coordinate (c9)  at (-1.90,-5.90);
    			\coordinate (c13) at ( 0.90,-5.90);
    			\coordinate (c16) at (-4.30,-6.10);
    			\coordinate (c18) at (-2.60,-5.00);
    			\coordinate (c10) at (-1.10,-6.70);
    			\coordinate (c15) at ( 0.10,-6.70);
    			\coordinate (c11) at (-1.90,-7.50);
    			\coordinate (c19) at ( 0.90,-7.50);
    			\coordinate (c14) at (-1.10,-8.30);
    			
    			\draw[normal edge]
    			(c1)--(c2)
    			(c2)--(c3)
    			(c1)--(c7)
    			(c7)--(c8)
    			(c2)--(c5)
    			(c5)--(c12)
    			(c5)--(c17)
    			(c12)--(c16)
    			(c17)--(c18);
    			
    			\draw[highlighted edge]
    			
    			(c3)--(c4)
    			(c4)--(c6)
    			(c6)--(c9)
    			(c6)--(c13)
    			(c9)--(c10)
    			(c10)--(c11)
    			(c11)--(c14)
    			(c13)--(c15)
    			(c15)--(c19);
    			
    			\node[tree vertex,label=above:{1}]              at (c1)  {};
    			\node[tree vertex,label=above left:{2}] at (c2)  {};
    			\node[tree vertex,label=above right:{7}]        at (c7)  {};
    			\node[highlighted vertex,label=above right:{3}] at (c3)  {};
    			\node[tree vertex,label=above right:{8}]        at (c8)  {};
    			\node[highlighted vertex,label=above right:{4}] at (c4)  {};
    			\node[tree vertex,label=above left:{5}]               at (c5)  {};
    			\node[highlighted vertex,label=above left:{6}]  at (c6)  {};
    			\node[tree vertex,label=left:{12}]              at (c12) {};
    			\node[tree vertex,label=above right:{17}]       at (c17) {};
    			\node[highlighted vertex,label=left:{9}]        at (c9)  {};
    			\node[highlighted vertex,label=above right:{13}] at (c13) {};
    			\node[tree vertex,label=below:{16}]             at (c16) {};
    			\node[tree vertex,label=below left:{18}]        at (c18) {};
    			\node[highlighted vertex,label=left:{10}] at (c10) {};
    			\node[highlighted vertex,label=right:{15}] at (c15) {};
    			\node[highlighted vertex,label=left:{11}]       at (c11) {};
    			\node[highlighted vertex,label=below right:{19}] at (c19) {};
    			\node[highlighted vertex,label=below right:{14}] at (c14) {};
    			
    			\node[highlighted name,anchor=west] at (0.35,-8.55) {\(L_3\)};
    			\node[tree name] at (0,-9.75) {\(T^{(2)}\)};
    			
    			\coordinate (second-arrow-end)   at (-5.55,-3.35);
    			\coordinate (third-arrow-start)  at ( 2.35,-3.35);
    		\end{scope}
    		
    		\begin{scope}[fourth tree]
    			\coordinate (d1)  at ( 0.00, 0.00);
    			\coordinate (d5)  at (-3.00,-3.00);
    			\coordinate (d2)  at ( 1.00,-1.00);
    			\coordinate (d3)  at ( 0.00,-2.00);
    			\coordinate (d7)  at ( 2.00,-2.00);
    			\coordinate (d4)  at ( 1.00,-3.00);
    			\coordinate (d8)  at ( 3.00,-3.00);
    			\coordinate (d6)  at (-0.50,-4.50);
    			\coordinate (d12) at (-4.40,-4.40);
    			\coordinate (d17) at (-2.40,-3.60);
    			\coordinate (d9)  at (-1.90,-5.90);
    			\coordinate (d13) at ( 0.90,-5.90);
    			\coordinate (d16) at (-3.50,-5.30);
    			\coordinate (d18) at (-1.80,-4.20);
    			\coordinate (d10) at (-1.10,-6.70);
    			\coordinate (d15) at ( 0.10,-6.70);
    			\coordinate (d11) at (-1.90,-7.50);
    			\coordinate (d19) at ( 0.90,-7.50);
    			\coordinate (d14) at (-1.10,-8.30);
    			
    			\draw[normal edge]
    			(d1)--(d5)
    			(d1)--(d2)
    			(d2)--(d3)
    			(d2)--(d7)
    			(d7)--(d8)
    			(d3)--(d4)
    			(d4)--(d6)
    			(d5)--(d12)
    			(d5)--(d17)
    			(d12)--(d16)
    			(d17)--(d18)
    			(d6)--(d9)
    			(d6)--(d13)
    			(d9)--(d10)
    			(d10)--(d11)
    			(d11)--(d14)
    			(d13)--(d15)
    			(d15)--(d19);
    			
    			\node[tree vertex,label=above:{1}]              at (d1)  {};
    			\node[tree vertex,label=above left:{5}]               at (d5)  {};
    			\node[tree vertex,label=above right:{2}]        at (d2)  {};
    			\node[tree vertex,label=above left:{3}]         at (d3)  {};
    			\node[tree vertex,label=above right:{7}]        at (d7)  {};
    			\node[tree vertex,label=above right:{4}]        at (d4)  {};
    			\node[tree vertex,label=above right:{8}]        at (d8)  {};
    			\node[tree vertex,label=above left:{6}]        at (d6)  {};
    			\node[tree vertex,label=left:{12}]              at (d12) {};
    			\node[tree vertex,label=above right:{17}]       at (d17) {};
    			\node[tree vertex,label=left:{9}]               at (d9)  {};
    			\node[tree vertex,label=above right:{13}]       at (d13) {};
    			\node[tree vertex,label=below:{16}]             at (d16) {};
    			\node[tree vertex,label=below left:{18}]        at (d18) {};
    			\node[tree vertex,label=left:{10}]        at (d10) {};
    			\node[tree vertex,label=right:{15}]       at (d15) {};
    			\node[tree vertex,label=left:{11}]              at (d11) {};
    			\node[tree vertex,label=below right:{19}]       at (d19) {};
    			\node[tree vertex,label=below right:{14}]       at (d14) {};
    			
    			\coordinate (third-arrow-end) at (-5.35,-3.35);
    			\node[tree name] at (0,-9.75) {\(\phi(T)\)};
    		\end{scope}
    		
    		\draw[transformation arrow]
    		(first-arrow-start) -- (first-arrow-end);
    		\draw[transformation arrow]
    		(second-arrow-start) -- (second-arrow-end);
    		\draw[transformation arrow]
    		(third-arrow-start) -- (third-arrow-end);
    	\end{tikzpicture}
    	
    	\caption{An example of the transformation $\phi$.}
    	\label{fig:phi}
    \end{figure}

   \begin{lemma}\label{lem-phi}
   	Given  a Jacobi tree $T\in \JB_S$,  we have \begin{equation}\label{eq-phi}
   		(\Lefttop, \first)\phi(T)=(\Lefttop, \first)T.
   	\end{equation} 
   	Moreover, the resulting tree $\phi(T)$ verifies the following  properties.
   	\begin{itemize}
   		\item[\upshape(\rmnum{1})] The root of the tree $\phi(T)$  is an Andr\'e  $\mathrm{I}$ node.
   		
   		\item[\upshape(\rmnum{2})]  The root of $\phi(T)$ is a Jacobi node if and only if the root of $T$ is an  Andr\'e $\mathrm{I}$ node.
   		\item[\upshape(\rmnum{3})]    Both the left subtree and the right subtree of the root in $\phi(T)$ are Jacobi trees.
   		\end{itemize}
   	\end{lemma}
   	 \begin{proof}
   	 	By the construction of $\phi(T)$, one can easily check that   the transformation $\phi$ preserves the statistics $\Lefttop$ and $\first$. Moreover, the tree $\phi(T)$ verifies property  (\rmnum{1}).  Now we proceed to prove (\rmnum{2}) and  (\rmnum{3}).   Here we use the same notations as in the defintion of $\phi$.

   	 	If the root $x$ is not an Andr\'e $\mathrm{I}$ node,  then by the construction of $\phi(T)$, the root $x$ of $\phi(T)$ has a left child $y$ and the right  chain starting from the node $y$  has an odd size. This ensures that the root $x$ of $\phi(T)$ must not  be a Jacobi node when the root $x$ is not an Andr\'e  $\mathrm{I}$ node. This completes the proof of (\rmnum{2}). 
		
		Now we proceed to prove (\rmnum{3}) by distinguishing  the following two cases. 
  \begin{itemize} 	 	
 \item{\bf Case 1: $z_1=x$.}
   	 	By the construction of $T^{(1)}$,   we have $\phi(T)=T^{(1)}$,  $\orc^*(T^{(1)})=1$ and the right chain starting with the node $y$ is odd.  Recall that the root $x$  has a left child $y$ in $T^{(1)}$. Therefore, we conclude that  both the left subtree and the right subtree of the root in $\phi(T)$ are Jacobi trees.
   	  	 	 
   	 	 \item{\bf Case 2: $z_1\neq x$.}
   	 	In this case,  we can generate a sequence of pairs $(T^{(2)}, z_2), (T^{(3)}, z_3),$ $ \ldots, (T^{(t)}, z_t)$ by the  iterated applications of $\beta$, where $z_t=x$.    By the construction of  $T^{(i)}$,   it is easy to check that $\orc^*(T^{(i)})=2$ for all $1\leq i<t$ and $\orc^*(T^{(t)})=1$. Moreover,  the right chain starting with the node $y$ is odd in $T^{(i)}$ for all $1\leq i\leq t$. Recall that the root $x$ is the parent of the node $y$ in $T^{(t)}$. 
		\end{itemize}
		Hence, we conclude that both the left subtree and the right subtree of the root in $\phi(T)$ are Jacobi trees, completing the proof.
   	 	\end{proof}
   	 	
   	 		\begin{lemma}\label{lem-phi-reverse}
   	 		For any two Jacobi trees $T$ and $T'$,  if $\phi(T)=\phi(T')$, then we have $T=T'$. 
   	 	\end{lemma}
   	 	\begin{proof}
   	 		
   	 		In order to prove the assertion, it  suffices to  show that for any tree $T\in \JB_S$, we can recover $T$ from $\phi(T)$.

   	 		If the root of $\phi(T)$ is both   an Andr\'e   $\mathrm{I}$  node and a Jacobi node, then 
   	 		by (\rmnum{2}) of Lemma~\ref{lem-phi},  the root of $T$ is an Andr\'e  $\mathrm{I}$  node.  By the construction of $\phi(T)$, we have $\phi(T)=T$. In this case, we can recover $T$ from $\phi(T)$ by letting $T=\phi(T)$.

   	 		Otherwise,  the root of $T$ is not an Andr\'e   $\mathrm{I}$  node.  In this case,   find the topmost node $y$ on the left arm of $T$  such that the node $\max(T)$  belongs to the right subtree of  $y$.      
   	 		Find the node $z_0$ on the right chain starting with the node $y$ such that either $\max(T)=z_0$ or  the node $\max(T)$ belongs to the left subtree of the node $z_0$. Assume that when applying the map $\phi$ to $T$, we generate a sequence  of pairs $(T^{(0)}, z_0), (T^{(1)}, z_1)$, $\ldots, (T^{(k)}, z_k)$ , where $T^{(0)}=T$, $\phi(T)=T^{(k)}$, $(T^{(1)}, z_1)=\alpha(T^{(0)}, z_0)$, and  $(T^{(i+1)}, z_{i+1})=\beta(T^{(i)}, z_i)$ for all $1\leq i<k$.  In order to show that  we can recover $T$ from $\phi(T)$, it suffices to show that we can recover   $T^{(i)}$ from    $T^{(i+1)}$ for all $0\leq i<k$.

   	 		From the construction of each $T^{(i+1)}$,  one can easily check that the resulting tree $T^{(i+1)}$ verifies the following desired properties. 
   	 		\begin{itemize}
   	 			\item[\upshape{(\rmnum{1})}]  The node $z_{i+1}$ is the unique  node 
   	 			on the left arm of $T^{(i+1)}$ so that the node $\max(T^{(i+1)})$ belongs to the right subtree of $z_{i+1}$;
   	 			\item[\upshape{(\rmnum{2})}]   Either $\max(T^{(i+1)})=z_i$ or the node $\max(T^{(i+1)})$ belongs to the left subtree of the node $z_i$. 
   	 			\item[\upshape{(\rmnum{3})}]  The node $y$ is the left child of the node $z_{i+1}$.
   	 			\item[\upshape{(\rmnum{4})}] We have $z_i>y$ if $i=0$ and $z_i<y$ otherwise. 
   	 		\end{itemize}
   	 		
   	 		Now we are in the position to give a description of the procedure to recover the
   	 		tree $T^{(i)}$ from $T^{(i+1)}$.  
   	 		
   	 		\begin{itemize}
   	 			\item Find the unique node $u$ on the left arm of $T^{(i+1)}$ so that the node $\max(T^{(i+1)})$ belongs to the right subtree of the node $u$.   Assume that the node $u$ has the left child $v$ in $T^{(i+1)}$.
   	 			\item Find the node $w$ on the right chain starting with the node $u$ so that  either $w=\max(T^{(i+1)})$ or the node $\max(T^{(i+1)})$ belongs to the left subtree of the node $w$.  
   	 			\item Generate a tree by distinguishing the following two cases.
   	 			\begin{itemize}
   	 				\item {\bf Case 1: $w>v$. } 
   	 				\begin{itemize}
   	 					\item Let $R$ be the subtree obtained from $T^{(i+1)}_{w}$ by removing the  right subtree of $w$.
   	 		 Remove the subtree $R$ from $T^{(i+1)}$.
   	 					\item Attach the subtree $R$ to the resulting tree by  inserting the node $w$ onto the right chain starting from the node $v$ so that the labels along this chain increase from left to right.
   	 				\end{itemize}
   	 				
   	 				\item {\bf Case 2: $w<v$.}
   	 				\begin{itemize}
   	 					\item Let $R$ be the  left subtree of the node $w$ in $T^{(i+1)}$.
   	 				 Remove the subtree $R$ from $T^{(i+1)}$.
   	 					\item  Insert the node $w$ onto the left arm of $T$ so that the labels along this chain increase  from top to bottom.
   	 					\item Attach   the subtree $R$ to the node $w$ as a right subtree of the node $w$ in the resulting tree. 
   	 				\end{itemize}
   	 				
   	 			\end{itemize}
   	 		\end{itemize}
   	 		Properties (\rmnum{1})-(\rmnum{4}) of $T^{(i+1)}$ ensure that when applying the above procedure to $T^{(i+1)}$,  we will have $u=z_{i+1}$, $v=y$ and $w=z_i$.   This guarantees that the above procedure reverses the transformation from $T^{(i)}$ to $T^{(i+1)}$. Therefore, applying the above procedure to $T^{(i+1)}$ yields the original tree $T^{(i)}$. This completes the proof.  
   	 		\end{proof}

		{\bf The construction of the  map $\Phi$.}
   		Given a Jacobi tree $T\in \JB_S$ with root $x$,    let $L$ and $R$ be the left subtree and the right subtree of  the root of $\phi(T)$, respectively. Let $\Phi(T)$ be the tree obtained from $\phi(T)$ by replacing  the subtree $L$ with $\Phi(L)$ and replacing the subtree $R$ with $\Phi(R)$. See Fig.~\ref{fig:Phi-RL-decomposition} for an example. 
   		
   	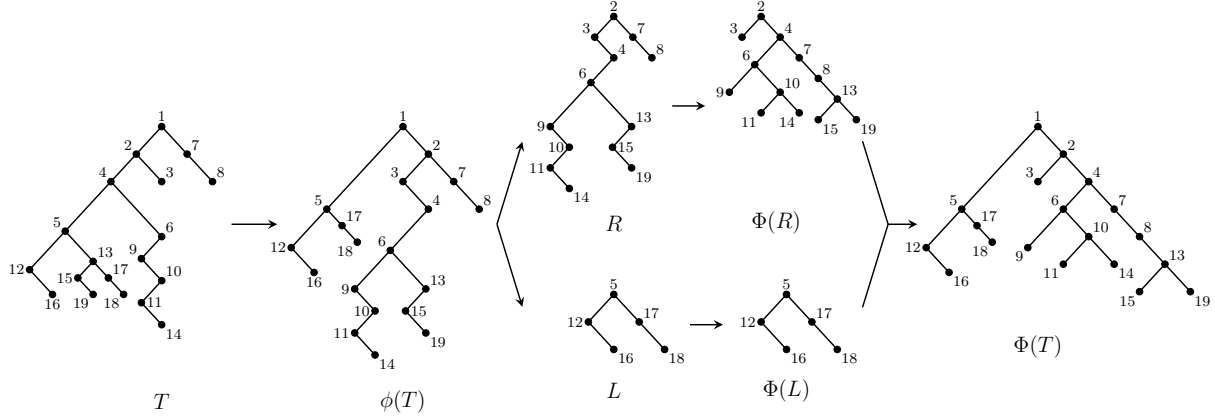
\begin{figure}
   		\centering
   		
   		\begin{tikzpicture}[
   			scale=0.7,
   			transform shape,
   			x=0.48cm,
   			y=0.52cm,
   			tree vertex/.style={
   				circle,
   				fill=black,
   				draw=black,
   				inner sep=0pt,
   				minimum size=3.5pt
   			},
   			every label/.style={
   				font=\scriptsize,
   				text=black,
   				inner sep=0.7pt
   			},
   			normal edge/.style={
   				draw=black,
   				line width=0.55pt,
   				line cap=round
   			},
   			transformation arrow/.style={
   				->,
   				draw=black,
   				line width=0.55pt,
   				>=stealth
   			},
   			tree name/.style={
   				font=\normalsize
   			},
   			%
   			source tree/.style={
   				shift={(0.00,0.00)}
   			},
   			second full tree/.style={
   				shift={(9.50,0.00)}
   			},
   			R tree/.style={
   				shift={(17.80,4.0)}
   			},
   			L tree/.style={
   				shift={(17.80,-6.10)}
   			},
   			Phi R tree/.style={
   				shift={(23.60,4.0)}
   			},
   			Phi L tree/.style={
   				shift={(24.60,-6.10)}
   			},
   			final tree/.style={
   				shift={(34.50,0.00)}
   			}
   			]
   			
   			\begin{scope}[source tree]
   				\coordinate (a1)  at ( 0.00, 0.00);
   				\coordinate (a2)  at (-1.00,-1.00);
   				\coordinate (a7)  at ( 1.00,-1.00);
   				\coordinate (a4)  at (-2.00,-2.00);
   				\coordinate (a3)  at ( 0.00,-2.00);
   				\coordinate (a8)  at ( 2.00,-2.00);
   				\coordinate (a5)  at (-3.80,-3.80);
   				\coordinate (a6)  at ( 0.00,-4.00);
   				\coordinate (a12) at (-5.20,-5.20);
   				\coordinate (a13) at (-2.70,-4.90);
   				\coordinate (a9)  at (-0.80,-4.80);
   				\coordinate (a16) at (-4.30,-6.10);
   				\coordinate (a15) at (-3.30,-5.50);
   				\coordinate (a17) at (-2.10,-5.50);
   				\coordinate (a10) at ( 0.00,-5.60);
   				\coordinate (a19) at (-2.70,-6.10);
   				\coordinate (a18) at (-1.50,-6.10);
   				\coordinate (a11) at (-0.80,-6.40);
   				\coordinate (a14) at ( 0.00,-7.20);
   				
   				\draw[normal edge]
   				(a1)--(a2)
   				(a1)--(a7)
   				(a2)--(a4)
   				(a2)--(a3)
   				(a7)--(a8)
   				(a4)--(a5)
   				(a4)--(a6)
   				(a5)--(a12)
   				(a5)--(a13)
   				(a12)--(a16)
   				(a13)--(a15)
   				(a13)--(a17)
   				(a15)--(a19)
   				(a17)--(a18)
   				(a6)--(a9)
   				(a9)--(a10)
   				(a10)--(a11)
   				(a11)--(a14);
   				
   				\node[tree vertex,label=above:{1}]               at (a1)  {};
   				\node[tree vertex,label=above left:{2}]          at (a2)  {};
   				\node[tree vertex,label=above right:{7}]         at (a7)  {};
   				\node[tree vertex,label=above left:{4}]          at (a4)  {};
   				\node[tree vertex,label=above right:{3}]         at (a3)  {};
   				\node[tree vertex,label=above right:{8}]         at (a8)  {};
   				\node[tree vertex,label=above left:{5}]          at (a5)  {};
   				\node[tree vertex,label=above right:{6}]         at (a6)  {};
   				\node[tree vertex,label=left:{12}]               at (a12) {};
   				\node[tree vertex,label=above right:{13}]        at (a13) {};
   				\node[tree vertex,label=above left:{9}]          at (a9)  {};
   				\node[tree vertex,label=below:{16}]              at (a16) {};
   				\node[tree vertex,label=left:{15}]               at (a15) {};
   				\node[tree vertex,label=above right:{17}]        at (a17) {};
   				\node[tree vertex,label=above right:{10}]        at (a10) {};
   				\node[tree vertex,label=below left:{19}]         at (a19) {};
   				\node[tree vertex,label=below left:{18}]         at (a18) {};
   				\node[tree vertex,label=right:{11}]              at (a11) {};
   				\node[tree vertex,label=below right:{14}]        at (a14) {};

   			\end{scope}
   			
   			\begin{scope}[second full tree]
   				\coordinate (b1)  at ( 0.00, 0.00);
   				\coordinate (b5)  at (-3.00,-3.00);
   				\coordinate (b2)  at ( 1.00,-1.00);
   				\coordinate (b3)  at ( 0.00,-2.00);
   				\coordinate (b7)  at ( 2.00,-2.00);
   				\coordinate (b4)  at ( 1.00,-3.00);
   				\coordinate (b8)  at ( 3.00,-3.00);
   				\coordinate (b6)  at (-0.50,-4.50);
   				\coordinate (b12) at (-4.40,-4.40);
   				\coordinate (b17) at (-2.40,-3.60);
   				\coordinate (b9)  at (-1.90,-5.90);
   				\coordinate (b13) at ( 0.90,-5.90);
   				\coordinate (b16) at (-3.50,-5.30);
   				\coordinate (b18) at (-1.80,-4.20);
   				\coordinate (b10) at (-1.10,-6.70);
   				\coordinate (b15) at ( 0.10,-6.70);
   				\coordinate (b11) at (-1.90,-7.50);
   				\coordinate (b19) at ( 0.90,-7.50);
   				\coordinate (b14) at (-1.10,-8.30);
   				
   				\draw[normal edge]
   				(b1)--(b5)
   				(b1)--(b2)
   				(b2)--(b3)
   				(b2)--(b7)
   				(b7)--(b8)
   				(b3)--(b4)
   				(b4)--(b6)
   				(b5)--(b12)
   				(b5)--(b17)
   				(b12)--(b16)
   				(b17)--(b18)
   				(b6)--(b9)
   				(b6)--(b13)
   				(b9)--(b10)
   				(b10)--(b11)
   				(b11)--(b14)
   				(b13)--(b15)
   				(b15)--(b19);
   				
   				\node[tree vertex,label=above:{1}]               at (b1)  {};
   				\node[tree vertex,label=above left:{5}]          at (b5)  {};
   				\node[tree vertex,label=above right:{2}]         at (b2)  {};
   				\node[tree vertex,label=above left:{3}]          at (b3)  {};
   				\node[tree vertex,label=above right:{7}]         at (b7)  {};
   				\node[tree vertex,label=above right:{4}]         at (b4)  {};
   				\node[tree vertex,label=above right:{8}]         at (b8)  {};
   				\node[tree vertex,label=above left:{6}]          at (b6)  {};
   				\node[tree vertex,label=left:{12}]               at (b12) {};
   				\node[tree vertex,label=above right:{17}]        at (b17) {};
   				\node[tree vertex,label=left:{9}]                at (b9)  {};
   				\node[tree vertex,label=above right:{13}]        at (b13) {};
   				\node[tree vertex,label=below:{16}]              at (b16) {};
   				\node[tree vertex,label=below left:{18}]         at (b18) {};
   				\node[tree vertex,label=left:{10}]               at (b10) {};
   				\node[tree vertex,label=right:{15}]              at (b15) {};
   				\node[tree vertex,label=left:{11}]               at (b11) {};
   				\node[tree vertex,label=below right:{19}]        at (b19) {};
   				\node[tree vertex,label=below right:{14}]        at (b14) {};
   				
   			\end{scope}
   			
   			\begin{scope}[R tree]
   				\coordinate (r2)  at ( 0.00, 0.00);
   				\coordinate (r3)  at (-0.75,-0.75);
   				\coordinate (r7)  at ( 0.75,-0.75);
   				\coordinate (r4)  at ( 0.00,-1.50);
   				\coordinate (r8)  at ( 1.50,-1.50);
   				\coordinate (r6)  at (-0.90,-2.40);
   				\coordinate (r9)  at (-2.50,-4.00);
   				\coordinate (r13) at ( 0.70,-4.00);
   				\coordinate (r10) at (-1.75,-4.75);
   				\coordinate (r15) at (-0.05,-4.75);
   				\coordinate (r11) at (-2.50,-5.50);
   				\coordinate (r19) at ( 0.70,-5.50);
   				\coordinate (r14) at (-1.75,-6.25);
   				
   				\draw[normal edge]
   				(r2)--(r3)
   				(r2)--(r7)
   				(r3)--(r4)
   				(r7)--(r8)
   				(r4)--(r6)
   				(r6)--(r9)
   				(r6)--(r13)
   				(r9)--(r10)
   				(r10)--(r11)
   				(r11)--(r14)
   				(r13)--(r15)
   				(r15)--(r19);
   				
   				\node[tree vertex,label=above:{2}]               at (r2)  {};
   				\node[tree vertex,label=above left:{3}]          at (r3)  {};
   				\node[tree vertex,label=above right:{7}]         at (r7)  {};
   				\node[tree vertex,label=above right:{4}]         at (r4)  {};
   				\node[tree vertex,label=above right:{8}]         at (r8)  {};
   				\node[tree vertex,label=above left:{6}]          at (r6)  {};
   				\node[tree vertex,label=left:{9}]                at (r9)  {};
   				\node[tree vertex,label=above right:{13}]        at (r13) {};
   				\node[tree vertex,label=left:{10}]         at (r10) {};
   				\node[tree vertex,label=right:{15}]        at (r15) {};
   				\node[tree vertex,label=left:{11}]               at (r11) {};
   				\node[tree vertex,label=below right:{19}]        at (r19) {};
   				\node[tree vertex,label=below right:{14}]        at (r14) {};

   			\end{scope}
   			
   			\begin{scope}[L tree]
   				\coordinate (l5)  at ( 0.00, 0.00);
   				\coordinate (l12) at (-1.00,-1.00);
   				\coordinate (l17) at ( 1.00,-1.00);
   				\coordinate (l16) at ( 0.00,-2.00);
   				\coordinate (l18) at ( 2.00,-2.00);
   				
   				\draw[normal edge]
   				(l5)--(l12)
   				(l5)--(l17)
   				(l12)--(l16)
   				(l17)--(l18);
   				
   				\node[tree vertex,label=above:{5}]               at (l5)  {};
   				\node[tree vertex,label=left:{12}]               at (l12) {};
   				\node[tree vertex,label=above right:{17}]        at (l17) {};
   				\node[tree vertex,label=below right:{16}]        at (l16) {};
   				\node[tree vertex,label=below right:{18}]        at (l18) {};

   			\end{scope}
   			
   			\begin{scope}[Phi R tree]
   				\coordinate (p2)  at ( 0.00, 0.00);
   				\coordinate (p3)  at (-0.75,-0.75);
   				\coordinate (p4)  at ( 0.75,-0.75);
   				\coordinate (p6)  at (-0.25,-1.75);
   				\coordinate (p7)  at ( 1.50,-1.50);
   				\coordinate (p8)  at ( 2.25,-2.25);
   				\coordinate (p9)  at (-1.25,-2.75);
   				\coordinate (p10) at ( 0.75,-2.75);
   				\coordinate (p11) at ( 0.00,-3.50);
   				\coordinate (p14) at ( 1.50,-3.50);
   				\coordinate (p13) at ( 3.00,-3.00);
   				\coordinate (p15) at ( 2.25,-3.75);
   				\coordinate (p19) at ( 3.75,-3.75);
   				
   				\draw[normal edge]
   				(p2)--(p3)
   				(p2)--(p4)
   				(p4)--(p6)
   				(p4)--(p7)
   				(p6)--(p9)
   				(p6)--(p10)
   				(p10)--(p11)
   				(p10)--(p14)
   				(p7)--(p8)
   				(p8)--(p13)
   				(p13)--(p15)
   				(p13)--(p19);
   				
   				\node[tree vertex,label=above:{2}]               at (p2)  {};
   				\node[tree vertex,label=above left:{3}]          at (p3)  {};
   				\node[tree vertex,label=above right:{4}]         at (p4)  {};
   				\node[tree vertex,label=above left:{6}]          at (p6)  {};
   				\node[tree vertex,label=above right:{7}]         at (p7)  {};
   				\node[tree vertex,label=above right:{8}]         at (p8)  {};
   				\node[tree vertex,label=left:{9}]                at (p9)  {};
   				\node[tree vertex,label=above right:{10}]        at (p10) {};
   				\node[tree vertex,label=below left:{11}]         at (p11) {};
   				\node[tree vertex,label=below left:{14}]               at (p14) {};
   				\node[tree vertex,label=above right:{13}]        at (p13) {};
   				\node[tree vertex,label=below right:{15}]              at (p15) {};
   				\node[tree vertex,label=below right:{19}]        at (p19) {};

   			\end{scope}
   			
   			\begin{scope}[Phi L tree]
   				\coordinate (q5)  at ( 0.00, 0.00);
   				\coordinate (q12) at (-1.00,-1.00);
   				\coordinate (q17) at ( 1.00,-1.00);
   				\coordinate (q16) at ( 0.00,-2.00);
   				\coordinate (q18) at ( 2.00,-2.00);
   				
   				\draw[normal edge]
   				(q5)--(q12)
   				(q5)--(q17)
   				(q12)--(q16)
   				(q17)--(q18);
   				
   				\node[tree vertex,label=above:{5}]               at (q5)  {};
   				\node[tree vertex,label=left:{12}]               at (q12) {};
   				\node[tree vertex,label=above right:{17}]        at (q17) {};
   				\node[tree vertex,label=below right:{16}]        at (q16) {};
   				\node[tree vertex,label=below right:{18}]        at (q18) {};

   			\end{scope}
   			
   			\begin{scope}[final tree]
   				\coordinate (f1)  at ( 0.00, 0.00);
   				\coordinate (f5)  at (-3.00,-3.00);
   				\coordinate (f2)  at ( 1.00,-1.00);
   				\coordinate (f3)  at ( 0.00,-2.00);
   				\coordinate (f4)  at ( 2.00,-2.00);
   				\coordinate (f6)  at ( 1.00,-3.00);
   				\coordinate (f7)  at ( 3.00,-3.00);
   				\coordinate (f8)  at ( 4.00,-4.00);
   				\coordinate (f12) at (-4.40,-4.40);
   				\coordinate (f17) at (-2.40,-3.60);
   				\coordinate (f16) at (-3.50,-5.30);
   				\coordinate (f18) at (-1.80,-4.20);
   				\coordinate (f9)  at (-0.40,-4.40);
   				\coordinate (f10) at ( 2.00,-4.00);
   				\coordinate (f11) at ( 1.00,-5.00);
   				\coordinate (f14) at ( 3.00,-5.00);
   				\coordinate (f13) at ( 5.00,-5.00);
   				\coordinate (f15) at ( 4.00,-6.00);
   				\coordinate (f19) at ( 6.00,-6.00);
   				
   				\draw[normal edge]
   				(f1)--(f5)
   				(f1)--(f2)
   				(f2)--(f3)
   				(f2)--(f4)
   				(f4)--(f6)
   				(f4)--(f7)
   				(f7)--(f8)
   				(f5)--(f12)
   				(f5)--(f17)
   				(f12)--(f16)
   				(f17)--(f18)
   				(f6)--(f9)
   				(f6)--(f10)
   				(f10)--(f11)
   				(f10)--(f14)
   				(f8)--(f13)
   				(f13)--(f15)
   				(f13)--(f19);
   				
   				\node[tree vertex,label=above:{1}]               at (f1)  {};
   				\node[tree vertex,label=above left:{5}]          at (f5)  {};
   				\node[tree vertex,label=above right:{2}]         at (f2)  {};
   				\node[tree vertex,label=above left:{3}]          at (f3)  {};
   				\node[tree vertex,label=above right:{4}]         at (f4)  {};
   				\node[tree vertex,label=above left:{6}]          at (f6)  {};
   				\node[tree vertex,label=above right:{7}]         at (f7)  {};
   				\node[tree vertex,label=above right:{8}]         at (f8)  {};
   				\node[tree vertex,label=left:{12}]               at (f12) {};
   				\node[tree vertex,label=above right:{17}]        at (f17) {};
   				\node[tree vertex,label=below right:{16}]        at (f16) {};
   				\node[tree vertex,label=below left:{18}]         at (f18) {};
   				\node[tree vertex,label=below left:{9}]          at (f9)  {};
   				\node[tree vertex,label=above right:{10}]        at (f10) {};
   				\node[tree vertex,label=below left:{11}]         at (f11) {};
   				\node[tree vertex,label=below right:{14}]        at (f14) {};
   				\node[tree vertex,label=above right:{13}]        at (f13) {};
   				\node[tree vertex,label=below left:{15}]         at (f15) {};
   				\node[tree vertex,label=below right:{19}]        at (f19) {};
   				
   			\end{scope}
   			
   			%
   			\coordinate (first-arrow-start) at ( 2.75,-3.55);
   			\coordinate (first-arrow-end)   at ( 4.55,-3.55);
   			
   			\coordinate (split-hub)         at (13.20,-3.55);
   			\coordinate (R-arrow-in)        at (14.20,-0.55);
   			\coordinate (L-arrow-in)        at (14.20,-6.55);
   			
   			\coordinate (R-arrow-start)     at (20.10,0.75);
   			\coordinate (R-arrow-end)       at (21.40,0.75);
   			\coordinate (L-arrow-start)     at (20.8,-7.20);
   			\coordinate (L-arrow-end)       at (22.10,-7.20);
   			
   			\coordinate (PhiR-out)          at (27.60,-0.55);
   			\coordinate (PhiL-out)          at (27.60,-6.55);
   			\coordinate (merge-hub)         at (28.60,-3.55);
   			\coordinate (final-arrow-end)   at (29.75,-3.55);
   			
   			\draw[transformation arrow]
   			(first-arrow-start) -- (first-arrow-end);
   			
   			\draw[transformation arrow]
   			(split-hub) -- (R-arrow-in);
   			\draw[transformation arrow]
   			(split-hub) -- (L-arrow-in);
   			
   			\draw[transformation arrow]
   			(R-arrow-start) -- (R-arrow-end);
   			\draw[transformation arrow]
   			(L-arrow-start) -- (L-arrow-end);
   			
   			\draw[normal edge]
   			(PhiR-out) -- (merge-hub);
   			\draw[normal edge]
   			(PhiL-out) -- (merge-hub);
   			\draw[transformation arrow]
   			(merge-hub) -- (final-arrow-end);
   			
   			\begin{scope}[source tree]
   				\node[tree name] at (0,-10) {\(T\)};
   			\end{scope}
   			\begin{scope}[second full tree]
   				\node[tree name] at (0,-10) {\(\phi(T)\)};
   			\end{scope}
   			\begin{scope}[R tree]
   				\node[tree name] at (0,-7.5) {\(R\)};
   			\end{scope}
   			\begin{scope}[L tree]
   				\node[tree name] at (0,-3.5) {\(L\)};
   			\end{scope}
   			\begin{scope}[Phi R tree]
   				\node[tree name] at (0.6,-7.5) {\(\Phi(R)\)};
   			\end{scope}
   			\begin{scope}[Phi L tree]
   				\node[tree name] at(0,-3.5)  {\(\Phi(L)\)};
   			\end{scope}
   			\begin{scope}[final tree]
   				\node[tree name] at (0,-8) {\(\Phi(T)\)};
   			\end{scope}
   			
   		\end{tikzpicture}
   		
   		\caption{An example of the map $\Phi$.}
   		\label{fig:Phi-RL-decomposition}
   	\end{figure}

   		\begin{lemma}\label{lem-Phi}
   			Given a Jacobi tree $T\in \JB_S$,  we have $\Phi(T)\in \AB^{\mathrm{I}}_{S}$  with
   			\begin{equation}\label{eq:lefttop-first}
   		(\Lefttop,   \first)T=(\Lefttop,   \first))\Phi(T).
   			\end{equation}
   			\end{lemma}
   		\begin{proof}
   			We proceed to prove the assertion by induction on $|S|$.  When $|S|=1$, the tree $T$ has only one node and  we have $\Phi(T)=T$. In this case,  it is apparent that $\Phi(T)\in \AB^{\mathrm{I}}_{S}$  with~\eqref{eq:lefttop-first} holds. Now we assume that  we have $\Phi(T)\in \AB^{\mathrm{I}}_{S'}$  with~\eqref{eq:lefttop-first} holds   for any Jacobi tree $T\in \JB_{S'}$ with $|S'|<|S|$. 
   			Given a Jacobi tree $T\in \JB_{S}$ with root $x$,   the corresponding tree $\Phi(T)$ is obtained by the following procedure.
   			\begin{itemize}
   				\item Apply the transformation $\phi$ to $T$.
   				\item Replace the left subtree $L$ of $\phi(T)$ with $\Phi(L)$, and replace the right subtree $R$ of $\phi(T)$ with $\Phi(R)$. 
   				\end{itemize}
   			 By Lemma~\ref{lem-phi}, the node $x$ is an Andr\'e $\mathrm{I}$ node. By induction hypothsis, both $\Phi(L)$ and $\Phi(R)$ are Andr\'e $\mathrm{I}$ trees. Then by the construction of the $\Phi(T)$,  it follows that the resulting tree $\Phi(T)$ is an Andr\'e $\mathrm{I}$ tree.
   				
   				Now we proceed to show that $(\Lefttop, \first)T= (\Lefttop, \first)\Phi(T)$.  We need to distinguish  two cases.
   				\begin{itemize}
				\item {\bf Case 1: $\larm(T)=1$.} \\
   				In this case, we have $\phi(T)=T$, $\first(T)=x$ and 
   				$\Lefttop(T)=\Lefttop(\phi(T))=\Lefttop(R)$.  By the construction of $\Phi(T)$, we deduce that
   				$\first(\Phi(T))=x=\first(T)$ and $\Lefttop(\Phi(T))=\Lefttop(\Phi(R)$. Hence, we deduce that  
   				$$
   				\Lefttop(\Phi(T))=\Lefttop(\Phi(R))=\Lefttop(R)=\Lefttop(T),
   				$$
   				where the second equality follows by induction hypothesis.
   				
   				\item {\bf Case 2:   $\larm(T)>1$. }\\
   				By the construction of $\Phi(T)$, we deduce that
   				$$
   				\first(\Phi(T))=\first(\Phi(L))=\first(L)=\first(\phi(T)) =\first(T),
   				$$
   				where the second equality follows from the induction hypothesis, and the last equaltiy follows from (\ref{eq-phi}). 
   				Moreover, we have
   				\[
   				\begin{aligned}
   				\Lefttop(\Phi(T))&=\Lefttop(\Phi(L))\cup \Lefttop(\Phi(R))\cup \{x\}\\
   				&=\Lefttop(L)\cup \Lefttop(R)\cup \{x\}\\
   				&=\Lefttop(\phi(T)) \\
   				&=\Lefttop(T),
   				\end{aligned}
   				\]
   				where the second equality follows from the induction hypothesis, and the last equality  follows from (\ref{eq-phi}). 
				\end{itemize} 
				We have proved that~\eqref{eq:lefttop-first} holds by induction, which  completes the proof.
   			\end{proof}

   			We are in position to prove Theorem \ref{thm-Phi}.

   			  \begin{proof}[{\bf Proof of Theorem \ref{thm-Phi}}] From Lemmas~\ref{lem-phi-reverse} and~\ref{lem-Phi}, it follows that the map $\Phi$ is an injection  such that for any $T\in \JB_S$, we have 	$(\Lefttop, \first)\Phi(T)=(\Lefttop, \first)T$.  It remains to 
   			 to show that the map $\Phi$ is a surjection.   This can be justified by the fact that $|\JB_S|=|\AB^{\mathrm{I}}_S|$, completing the proof. 
			 \end{proof}
			 
			 We are ready for the proof of Theorem \ref{thm-main}. 
   			
   			\begin{proof}[{\bf Proof of Theorem \ref{thm-main}}]  In views of Lemmas~\ref{thm-xi} and~\ref{thm-theta} together with Theorem~\ref{thm-Phi}, the map $\Omega=\theta^{-1}\circ\Phi\circ\xi$ serves as the desired bijection between $\mathcal{J}_S$ and $\mathcal{A}_S$ such that 
   			for any $\pi\in \mathcal{J}_S$, we have $	(\Ascbot, \last)\pi= (\Desbot, \first)\Omega(\pi), $ completing the proof.
   \end{proof}

  \section{Proof of Theorem \ref{thm-main2}} 
  \label{Sec:4}
   		This section is devoted to the bijective proof of Theorem \ref{thm-main2}.  
		
		\subsection{Bijective proof of the second identity of~\eqref{eq-equidistribution}}
   		\begin{definition}[Andr\'e $\mathrm{II}$ trees]
		A binary tree $T\in\B_S$ is called an \blue{ Andr\'e $\mathrm{II}$ tree} on $S$ if the smallest
child of any internal node is its right child. 
   			Denote by $\AB^{\mathrm{II}}_S$ the set of all   Andr\'e $\mathrm{II}$ trees on $S$. 
			\end{definition} 		
\begin{figure}
  \centering
  \begin{tikzpicture}[
    scale=1,                 
    transform shape,
    x=0.62cm,                 
    y=0.68cm,                 
    tree vertex/.style={
      circle,
      fill=black,
      draw=black,
      inner sep=0pt,
      minimum size=3.5pt      
    },
    every label/.style={
      font=\small,            
      text=black,
      inner sep=0.8pt
    },
    normal edge/.style={
      draw=black,
      line width=0.55pt,      
      line cap=round
    },
    transformation arrow/.style={
      ->,
      draw=black,
      line width=0.55pt,
      >=stealth
    },
    rotation arrow/.style={
      ->,
      draw=red,
      line width=0.70pt,
      line cap=round,
      >=stealth
    },
    map label/.style={
      font=\normalsize
    },
    permutation/.style={
      font=\normalsize
    },
    first tree/.style={
      xshift=2.25cm
    },
    second tree/.style={
      xshift=9.3cm
    }
  ]
    \begin{scope}[first tree]
      \coordinate (a0)  at ( 0.00, 0.00);
      \coordinate (a2)  at (-1.50,-1.50);
      \coordinate (a1)  at ( 1.50,-1.50);
      \coordinate (a8)  at (-2.50,-2.50);
      \coordinate (a7)  at (-0.50,-2.50);
      \coordinate (a4)  at ( 0.50,-2.50);
      \coordinate (a3)  at ( 2.50,-2.50);
      \coordinate (a10) at (-3.50,-3.50);
      \coordinate (a9)  at (-1.50,-3.50);
      \coordinate (a6)  at ( 1.50,-3.50);
      \coordinate (a5)  at ( 3.50,-3.50);

      \draw[normal edge]
        (a0)--(a2)
        (a0)--(a1)
        (a2)--(a8)
        (a2)--(a7)
        (a8)--(a9)
        (a1)--(a4)
        (a1)--(a3)
        (a3)--(a6)
        (a3)--(a5);

      \node[tree vertex,label={[text=red]above:{0}}]             at (a0)  {};
      \node[tree vertex,label=above left:{2}]        at (a2)  {};
      \node[tree vertex,label={[text=red]above right:{1}}]       at (a1)  {};
      \node[tree vertex,label=left:{8}]              at (a8)  {};
      \node[tree vertex,label=below left:{7}]        at (a7)  {};
      \node[tree vertex,label=below right:{4}]       at (a4)  {};
      \node[tree vertex,label={[text=red]above right:{3}}]       at (a3)  {};
      \node[tree vertex,label=below right:{9}]       at (a9)  {};
      \node[tree vertex,label=below left:{6}]        at (a6)  {};
      \node[tree vertex,label={[text=red]below right:{5}}]       at (a5)  {};
    \end{scope}

    \begin{scope}[second tree]
      \coordinate (b1)  at ( 0.00, 0.00);
      \coordinate (b2)  at (-1.50,-1.50);
      \coordinate (b3)  at ( 1.50,-1.50);
      \coordinate (b8)  at (-2.50,-2.50);
      \coordinate (b7)  at (-0.50,-2.50);
      \coordinate (b4)  at ( 0.50,-2.50);
      \coordinate (b5)  at ( 2.50,-2.50);
      \coordinate (b10) at (-3.50,-3.50);
      \coordinate (b9)  at (-1.50,-3.50);
      \coordinate (b6)  at ( 1.50,-3.50);

      \draw[normal edge]
        (b1)--(b2)
        (b1)--(b3)
        (b2)--(b8)
        (b2)--(b7)
        (b8)--(b9)
        (b3)--(b4)
        (b3)--(b5)
        (b5)--(b6);

      \node[tree vertex,label={[text=red]above:{1}}]        at (b1)  {};
      \node[tree vertex,label=above left:{2}]               at (b2)  {};
      \node[tree vertex,label={[text=red]above right:{3}}]  at (b3)  {};
      \node[tree vertex,label=left:{8}]                     at (b8)  {};
      \node[tree vertex,label=below left:{7}]               at (b7)  {};
      \node[tree vertex,label=below right:{4}]              at (b4)  {};
      \node[tree vertex,label={[text=red]above right:{5}}]  at (b5)  {};
      \node[tree vertex,label=below right:{9}]              at (b9)  {};
      \node[tree vertex,label=below left:{6}]               at (b6)  {};
    \end{scope}

    \begin{scope}[x=1cm,y=1cm]
      \coordinate (rotation-one-start)   at (3.28,-0.67);
      \coordinate (rotation-one-control) at (3.34,-0.14);
      \coordinate (rotation-one-end)     at (2.46, 0.16);
      
      \coordinate (rotation-two-start)   at (4.08,-1.37);
      \coordinate (rotation-two-control) at (4.10,-1);
      \coordinate (rotation-two-end)     at (3.52,-0.85);
      
      \coordinate (rotation-three-start)   at (4.60,-2.32);
      \coordinate (rotation-three-control) at (4.8,-1.93);
      \coordinate (rotation-three-end)     at (4.14,-1.54);
      
      \coordinate (tree-arrow-start) at (6.25,-1.22);
      \coordinate (tree-arrow-end)   at (8.25,-1.22);
      \coordinate (theta-arrow-start) at (11.50,-2.90);
      \coordinate (theta-arrow-end)   at (11.50,-3.90);
      \coordinate (theta-label)       at (11.85,-3.30);
      \coordinate (permutation-label) at (9.8,-4.30);
      
      \draw[rotation arrow]
      (rotation-one-start)
      .. controls (rotation-one-control) ..
      (rotation-one-end);
      
      \draw[rotation arrow]
      (rotation-two-start)
      .. controls (rotation-two-control) ..
      (rotation-two-end);
      
      \draw[rotation arrow]
      (rotation-three-start)
      .. controls (rotation-three-control) ..
      (rotation-three-end);
      
      \node at  (6,-1){$\longrightarrow$};
      \node at  (6,-0.7){$\psi$};
      

    \end{scope}
  \end{tikzpicture}

  \caption{An example of  $\psi$.}
  \label{fig:tree-psi}
\end{figure}
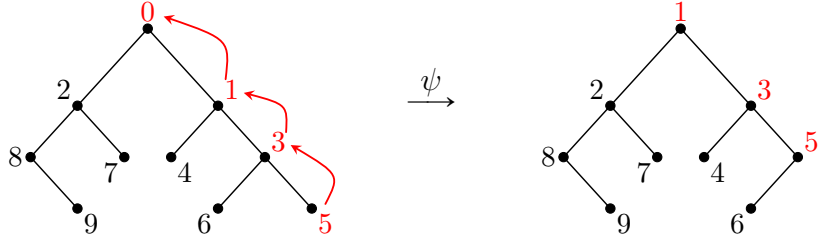

   		For example,  Fig.~\ref{fig:tree-psi} (in left) depicts  an Andr\'e  $\mathrm{II}$ tree.
   		 Given a binary tree $T\in \B_S$, let $\l(T)$ denotes the number of left edges of $T$   and let $\last(T)$ denote the label of the rightmost  node on the right arm  of $T$.  For example, if we let $T$ be the tree displayed in Fig.~\ref{fig:tree-psi} (in left), then  $\l(T)=4$ and $\last(T)=5$. 
   		The following properties of the bijection $\theta$ is clear from its construction.

   		\begin{lemma}\label{lem-theta-andre2}
   			Fix a set $S$ of positive integers.  The map $\theta$ restricts to a bijection between 
   			$\mathcal{A}^{\mathrm{II}}_S$ and $\AB^{\mathrm{II}}_S$ such that for any $\pi \in \mathcal{A}^{\mathrm{II}}_S$, we have
   			$$
   			\last(\theta(\pi))=\last(\pi)\quad\text{and}\quad \l(\theta(\pi))=\des(\pi). 
   			$$
   			\end{lemma}

   	   \begin{definition}[Simsun trees]
   		A binary tree $T\in\B_n$ is called a \blue{simsun tree} if the smallest
child of any internal node not on the right chain must be a right child.  Denote by $\mathsf{ST}_n$ the set of all   simsun trees on $[n]$.
			\end{definition} 	
			
			It was observed in~\cite[Proposition~5.2]{Lin} that $\theta$ restricts to a bijection between $\mathcal{RS}_n$ and $\mathsf{ST}_n$. Moreover, the following feature  of the bijection $\theta$ is clear from its construction. 
			
			\begin{lemma}\label{lem-theta-simsun}
   			  The map $\theta$ restricts to a bijection between 
   			$\mathcal{RS}_n$ and $\mathsf{ST}_n$ such that for any $\pi \in \mathcal{RS}_n$, we have
   			$$
   			\last(\theta(\pi))=\last(\pi)\quad\text{and}\quad \l(\theta(\pi))=\des(\pi). 
   			$$
   			\end{lemma}
   		  
   			In order to prove the second identity of~\eqref{eq-equidistribution}, we construct a bijection  $\psi: \AB^{\mathrm{II}}_{[n]\cup\{0\}}\rightarrow\mathsf{ST}_n$ in the spirt of  Sch\" utzenberger's {\em jeu de taquin}.  Given $T\in\AB^{\mathrm{II}}_{[n]\cup\{0\}}$, remove the label $0$ of the
root and then successively move up the right child  of the node
without label. This procedure ends until the label of the  leaf in the right chain has been
moved up. We remove this leaf without a label, which results in the desired tree $\psi(T)$. The reader is referred to Fig.~\ref{fig:tree-psi}  for an example of $\psi$. 
   			 It is plain to see that  $\psi$ is  a bijection satisfies the 
   		 following desired properties.
   		 \begin{lemma}\label{lem-psi}
   		 	For $n\geq 1$,  $\psi: \AB^{\mathrm{II}}_{[n]\cup\{0\}}\rightarrow\mathsf{ST}_n$ is a bijection  such that for any $T\in \AB^{\mathrm{II}}_{[n]\cup \{0\}}$, we have
   		 	$$
   		 	\last(\psi(T))=\last(T)\,\, \mbox{and}\,\, \l(\psi(T))=\l(T).
   		 	$$
   		 	\end{lemma}
   		 	
   		 \begin{proof}[{\bf Proof of the second identity of~\eqref{eq-equidistribution}}]	For $n\geq 2$,   given a tree $T\in  \AB^{\mathrm{II}}_{[n]}$,  we can construct a tree $\lambda(T)\in \AB^{\mathrm{II}}_{[n-1]\cup{0}}$ by decreasing the label  of each node by one.   Clearly, the map $\lambda$ is a bijection  between $\mathcal{AB}^{\mathrm{II}}_{[n]}$ and $\mathcal{AB}^{\mathrm{II}}_{[n-1]\cup \{0\}}$.
		 In view of Lemmas \ref{lem-theta-andre2} and \ref{lem-psi},  the composition  $\Theta:=\theta^{-1}\circ\psi\circ\lambda\circ\theta$ is a bijection between 
   		  	$\mathcal{A}^{\mathrm{II}}_{n}$ and $\mathcal{RS}_{n-1}$  such that for any 
   		  	$\pi\in \mathcal{A}^{\mathrm{II}}_{n}$, we have 
   		  	$$
   		  	\last(\pi)=\last(\Theta(\pi))-1 \quad\mbox{and}\quad\des(\pi)=\des(\Theta(\pi)).
   		  	$$
	This proves  the second identity of~\eqref{eq-equidistribution}. 	 
		  \end{proof}

   	\subsection{Bijective proof of the first identity of~\eqref{eq-equidistribution}}

   			We next describe a bijection $\Gamma$ between Andr\'e $\mathrm{I}$ and Andr\'e  $\mathrm{II}$
   			permutations introduced by  Foata and Han~\cite{FH16}.  The bijection $\Gamma$ consists of two bijective maps $g\colon \mathcal{A}_n^{\mathrm{I}}\to \mathcal{A}_n^{\mathrm{I}}$ and $h\colon \mathcal{A}_n^{\mathrm{I}}\to \mathcal{A}_n^{\mathrm{II}}$.
   			
   			For a nonempty permutation \(\pi=\pi_1\pi_2\cdots \pi_n\), set \(x_{n+1}=0\). A  letter $\pi_i$ ($1\leq i\leq n)$ is called a \blue{\em left-to-right  minimum} if $\pi_j>\pi_i$ for all $j<i$.  
   			The \blue{\em spike} of $\pi$, denoted by \(\spi(\pi)\), is the letter \(\pi_i\),
   			where \(i\) is determined by
   			\[
   			\pi_j\geq \pi_1\quad(1\leq j\leq i)
   			\qquad\text{and}\qquad
   			\pi_{i+1}<\pi_1.
   			\]
   		Since 	\(x_{n+1}=0\), \(\spi(\pi)\) always exists. 
		
   			We first recall the map \(g\). Let
   			\(\pi\in\mathcal{A}_n^{\mathrm{I}}\), and let
   			\(1=i_1<i_2<\cdots<i_r\) be the positions of the left-to-right minimum  of
   			\(\pi\). 
   			Set \(i_{r+1}=n+1\) and let
   			\[
   			u_j:=\pi_{i_j}\pi_{i_j+1}\cdots\pi_{i_{j+1}-1}
   			\qquad (1\leq j\leq r).
   			\]
   			Thus, \(\pi\) admits the following canonical factorization 
   			\[
   			\pi=u_1u_2\cdots u_r.
   			\]
   			  For a permutation 
   			\(u=x_1x_2\cdots x_m\), define
   			\[
   			\tilde{u}:=(n+1-x_m)(n+1-x_{m-1})\cdots(n+1-x_1).
   			\]
   			The map \(g\) is given by
   			\[
   			g(\pi):=\tilde{u}_1\tilde{u}_2\cdots\tilde{u}_r.
   			\]
   		For example, if  
   			$
   			\pi=6587(10)13294(11)
   			\in\mathcal{A}_{11}^{\mathrm{I}},
   			$
   			then its canonical factorization is given by 
   			\[
   			\pi=\blue{\bf6}\mid \blue{\bf5}\,8\,7\,10\mid \blue{\bf1}\,3\,2\,9\,4\,11.
   			\]
   			By the  definition of \(g\), we have $g(\pi)=\blue{\bf6}\,2\,5\,4\,\blue{\bf7}\,1\,8\,3\,10\,9\,\blue{\bf11}$.

   			We next recall the recursive map \(h\). Let
   			\(S=\{a_1<a_2<\cdots<a_m\}\) be a finite set of positive integers. Set
   			\[
   			h(\emptyword)=\emptyword,
   			\qquad h(a_1)=a_1,
   			\qquad h(a_1a_2)=a_1a_2,
   			\]
   			where $\emptyword$ denotes the empty permutation. 
   			For \(m\geq3\), any \(w\in\mathcal{A}_S^{\mathrm{I}}\) has one of the
   			following two forms:
   			\[
   			\text{(i)}\quad w=v_0a_1v_1a_2v_2,
   			\qquad
   			\text{(ii)}\quad w=v_0a_2v_2a_1v_1.
   			\]
   			Here \(v_0,v_1,v_2\), and \(v_0a_2v_2\) are Andr\'e I permutations, 
   		and \(v_0\) may be empty. In case~(i),
   			\(v_2\) is nonempty, whereas in case~(ii), \(v_1\) is nonempty. In both
   			cases, define
   			\[
   			h(w):=h(v_1)a_1h(v_0a_2v_2).
   			\]
   			For example, 	successive applications of the recursive definition of \(h\) yield that
   			\[
   			\begin{aligned}
   				h(6\;2\;5\;4\;7\;1\;8\;3\;10\;9\;11)
   				&=h(8\;3\;10\;9\;11)\,1\,
   				h(6\;2\;5\;4\;7),\\
   				h(8\;3\;10\;9\;11)
   				&=h(10\;9\;11)\,3\,h(8)
   				=11\;9\;10\;3\;8,\\
   				h(6\;2\;5\;4\;7)
   				&=h(5)\,2\,h(6\;4\;7)
   				=5\;2\;7\;4\;6.
   			\end{aligned}
   			\]

   			  \begin{proof}[{\bf Proof of the first identity of~\eqref{eq-equidistribution}}]
   			Set $\Gamma:=h\circ g$. We aim to show that the bijection $\Gamma: \mathcal{A}_n^{\mathrm{I}}\rightarrow\mathcal{A}_n^{\mathrm{II}}$ satisfies 
   				\begin{equation}\label{prop:Gamma}
   				\last(\Gamma(\pi))=n+1-\first(\pi)
   				\qquad\text{and}\qquad
   				\des(\Gamma(\pi))=\des(\pi).
   				\end{equation}
	
   				For \(n\leq2\), the assertion is immediate. Suppose that \(n\geq3\).
   				By~\cite[Eqs.~(3.3) and (4.2)]{FH16}, we have 
   				\[
   				\last(h(w))=\spi(w)
   				\quad\text{and}\quad
   				\spi(g(\pi))=n+1-\first(\pi).
   				\]
   				Consequently,
   				\[
   				\last(\Gamma(\pi))
   				=\last(h(g(\pi)))
   				=\spi(g(\pi))
   				=n+1-\first(\pi).
   				\]
   				
   				It remains to show that \(\Gamma\) preserves the statistic $\des$. We
   				first consider \(g\). Here we use the same notations in the definition of $g$.   Write \(\pi=u_1u_2\cdots u_r\), then
   				\[\des(\pi)=r-1+\sum\limits_{j=1}^{r}\des(u_j).\]
				 By the construction of $\tilde u_j$, we have $\des(u_j)=\des(\tilde u_j)$.  In order to prove that $\des(\pi)=\des(g(\pi))$, it remains to verify that $\last(\tilde u_j)>\first(\tilde u_{j+1})$ for all $1\leq j<r$. Recall that 
   				$\last(\tilde u_j)=n+1-\pi_{i_j}$ and $\first(\tilde u_{j+1})=n+1-\pi_{i_{j+2}-1}$. Since  $\pi$ is an Andr\'e $\mathrm{I}$ permutation,  we have $\pi_k<\pi_{i_{j+1}-1}$ for all $k<i_{j+2}-1$. This implies that   $\pi_{i_j}<\pi_{i_{j+2}-1}$. Therefore, we deduce that  $\last(\tilde u_j)>\last(\tilde u_{j+1})$  and hence
   			$\des(g(\pi))=\des(\pi)$ holds. 
   				
   				We next prove that \(h\) preserves the statistic $\des$ by induction on the length
   				of \(w\). Here we again adopt the same notations in the definition of $h$. This is clear for any word of length at most two. For a word of
   				length at least three,  if $w=v_0a_1v_1a_2v_2$, then we have $h(w)=h(v_1)a_1h(v_0a_2v_2)$. 
   				Let $\chi(A)=1$ if the statement $A$ is true, and $\chi(A)=0$ otherwise.
   				We have
				   \[
   				\begin{aligned}
   				\des(w)&=\des(v_0)+\des(v_1)+\des(v_2)+\chi(v_0\neq \emptyword)+\chi(v_1\neq \emptyword)\\
   				&= \des(v_1)+\des(v_0a_2v_2)+\chi(v_1\neq \emptyword)\\
   				&= \des(h(v_1))+\des(h(v_0a_2v_2))+\chi(v_1\neq \emptyword)\\
   				&= \des(h(w)),
   				\end{aligned}
   				\]
   				where  the third equality follows from the induction hypothesis.  
   			If $w=v_0a_2v_2a_1v_1$,  where $v_1\neq \emptyword$, then we have $h(w)=h(v_1)a_1h(v_0a_2v_2)$.  In this case, we have
   				\[
   				\begin{aligned}
   					\des(w)&=\des(v_0)+\des(v_1)+\des(v_2)+\chi(v_0\neq \emptyword)+1\\
   					&= \des(v_1)+\des(v_0a_2v_2)+1\\
   					&= \des(h(v_1))+\des(h(v_0a_2v_2))+1 \\
   					&=\des(h(w)),
   				\end{aligned}
   				\]
   				where  the third equality follows from the induction hypothesis.   Thus, we have $\des(w)=\des(h(w))$ in all cases. This proves~\eqref{prop:Gamma}, which implies the first identity of~\eqref{eq-equidistribution}. 
   			\end{proof}

  \section{Proof of Theorem~\ref{thm:jacobi-F-closed} }
  \label{Sec:5}	 
   	This section aims to prove Theorem~\ref{thm:jacobi-F-closed}. 		
   				Let 
$$
D_{\mathrm e}(t;z):=\sum\limits_{n\geq0}D_{2n}(t)\frac{z^{2n}}{(2n)!} 
  \quad\text{and}\quad D_{\mathrm o}(t;z):=\sum\limits_{n\geq0}D_{2n+1}(t)\frac{z^{2n+1}}{(2n+1)!}.
$$  Set
   			\[
   			Q(t;z):=1+t\bigl(D_{\mathrm e}(t;z)-1\bigr).
   			\]
   			When there is no  confusion, we suppress $t$ and $z$,  and  a prime  denotes
   			differentiation with respect to $z$.
   			
   			\begin{lemma}
   				\label{lem:jacobi-D}
   				The following relationships hold: 
   				\begin{equation}
   					D'=DQ,
   					\qquad
   					Q'=tD_{\mathrm o}Q,
   					\qquad
   					D_{\mathrm e}^{2}-D_{\mathrm o}^{2}=1.
   					\label{eq:jacobi-D-relations}
   				\end{equation}
   			\end{lemma}
   			
   			\begin{proof}
   				Every nonempty Jacobi permutation $\pi$ has a unique decomposition
   				$
   				\pi=u\,1\,v,
   				$
   				where   $u$ and $v$ are Jacobi permutations and
   				$|v|$ is even.    Suppose
   				that $|u|=p$, $|v|=q$, and $p+q=n$.  The letters
   				other than $1$ may be assigned to $u$ and $v$ in $\binom np$ ways.   If $v$ is empty, we have $\asc(\pi)=\asc(u)$.  If $v$ is nonempty, we have 
   				$\asc(\pi)=\asc(u)+\asc(v)+1$.   It follows that
   				\[
   				D_{n+1}(t)
   				=D_n(t)
   				+t\!\sum_{\substack{p+q=n\\q\geq2\;\mathrm{even}}}
   				\binom np D_p(t)D_q(t).
   				\]
   				Mutiplying both sides  by $z^n/n!$ and summing over all  $n\geq0$, this
   				recurrence gives 
   				\[
   				D'=D\bigl[1+t(D_{\mathrm e}-1)\bigr]=DQ.
   				\]
   				Since 
   				$Q$
   				is an even function of 
   				$z$, 
   				comparing the odd and even parts of both sides of the above equation yields
   				\[
   				D_{\mathrm e}'=D_{\mathrm o}Q,
   				\qquad
   				D_{\mathrm o}'=D_{\mathrm e}Q.
   				\]
   				Therefore, we deduce that 
   				\begin{align*}
   					&Q'=tD_{\mathrm e}'=tD_{\mathrm o}Q\quad\text{and}\\
   					&\bigl(D_{\mathrm e}^{2}-D_{\mathrm o}^{2}\bigr)'
   					=2D_{\mathrm e}D_{\mathrm o}Q-2D_{\mathrm o}D_{\mathrm e}Q=0.
   				\end{align*}
   				Since $D_{\mathrm e}(t;0)=1$ and $D_{\mathrm o}(t;0)=0$, we have
   				$D_{\mathrm e}^{2}-D_{\mathrm o}^{2}=1$, completing the proof.  
   			\end{proof}
			
			We are ready to prove Theorem~\ref{thm:jacobi-F-closed}. 
   			
   			\begin{proof}[{\bf Proof of Theorem~\ref{thm:jacobi-F-closed}}] 
   			Every nonempty Jacobi permutation $\pi$ has a unique decomposition
   			$
   			\pi=u\,1\,v,
   			$
   			where   $u$ and $v$ are Jacobi permutations and
   			$|v|$ is even.  When $v$ is empty, we have $\pi=u1$, $\asc(\pi)=\asc(u)$ and $\last(\pi)=1$. This yields that
   			\begin{equation}
   				F(t;0,y)=D(t;y).
   				\label{eq:jacobi-F-boundary}
   			\end{equation}
   			
   			Now suppose that $v$ is nonempty with $\last(v)=k$ for some  $k\geq 2$.  In this case, we have $\asc(\pi)=\asc(u)+\asc(v)+1$ and $\last(\pi)=\last(v)$.
   			Besides $k$, suppose
   			that $v$ contains $i$ letters  from $\{2,\ldots,k-1\}$ and $j$ letters  from
   			$\{k+1,\ldots,n\}$.  Then $|v|=i+j+1$.  Thus $i+j$ is odd, since $v$
   			has even length.  The letters of $v$ can be chosen in
   			$\binom{k-2}{i}\binom{n-k}{j}$ ways.  After standardization, $v$ ends
   			with $i+1$, while the remaining $n-i-j-2$ letters form $u$.  Recall that the \blue{\em standardization }of $v$ is the permutations obtained from $v$ by replacing  its $\ell$-th smallest letter by $\ell$ for all $\ell\geq1$.   Therefore, we have
   			\begin{equation}
   				 J_{n,k}(t)
   				=t\!\sum_{\substack{0\leq i\leq k-2,\;0\leq j\leq n-k\\
   						i+j\;\mathrm{odd}}}
   				\binom{k-2}{i}\binom{n-k}{j}
   				 J_{i+j+1,i+1}(t)D_{n-i-j-2}(t).
   				\label{eq:jacobi-J-recurrence}
   			\end{equation}
   			
   			By the binomial theorem, we have
   			\[
   			D(t;x+y)
   			=\sum_{p,q\geq0}D_{p+q}(t)
   			\frac{x^p}{p!}\frac{y^q}{q!}.
   			\]
   			Let
   			$$
   				H(t;x,y)
   				:=\frac{F(t;x,y)-F(t;-x,-y)}2=\sum_{\substack{i,j\geq0\\i+j\;\mathrm{odd}}}
   			      J_{i+j+1,i+1}(t)\frac{x^i}{i!}\frac{y^j}{j!}.
   			$$
   			The coefficient of
   			$x^{k-2}y^{n-k}/((k-2)!(n-k)!)$ in $D(t;x+y)H(t;x,y)$ is the sum on the
   			right-hand side of~\eqref{eq:jacobi-J-recurrence} without the factor $t$.  The
   			same coefficient in ${\partial F(t;x,y)\over \partial x}$ is given by
   			$J_{n,k}(t)$.  Therefore, we deduce that
   			\begin{equation}
   				\frac{\partial F(t;x,y)}{\partial x} 
   				=tD(t;x+y)H(t;x,y).
   				\label{eq:jacobi-F-differential}
   			\end{equation}
   			By (\ref{eq:jacobi-F-differential}), it follows that
   			\[
   			\frac{\partial H(t;x,y)}{\partial x} 
   			=tD_{\mathrm o}(t;x+y)H(t;x,y).
   			\]
   			The identity $Q'=tD_{\mathrm o}Q$ in \eqref{eq:jacobi-D-relations} gives
   			\[
   			\frac{\partial Q(t;x+y)}{\partial x}
   			=tD_{\mathrm o}(t;x+y)Q(t;x+y).
   			\]
   			Hence $H(t;x,y)/Q(t;x+y)$ is independent of $x$.  From
   			\eqref{eq:jacobi-F-boundary},  it follows that $H(t;0,y)=D_{\mathrm o}(t;y)$ and
   			\begin{equation}
   				H(t;x,y)
   				=D_{\mathrm o}(t;y)\frac{Q(t;x+y)}{Q(t;y)}.
   				\label{eq:jacobi-H-formula}
   			\end{equation}
   			Using \eqref{eq:jacobi-H-formula} and $D'=DQ$ in
   			\eqref{eq:jacobi-F-differential}, we have
   			\[
   			\frac{\partial F(t;x,y)}{\partial x} 
   			=\frac{tD_{\mathrm o}(t;y)}{Q(t;y)}D'(t;x+y).
   			\]
   			Integrating from $0$ to $x$ and using \eqref{eq:jacobi-F-boundary}, we obtain
   			\begin{equation}
   				F(t;x,y)
   				=D(t;y)
   				+\frac{tD_{\mathrm o}(t;y)}{Q(t;y)}
   				\bigl(D(t;x+y)-D(t;y)\bigr).
   				\label{eq:jacobi-F-in-D}
   			\end{equation}
   			Since $D_{\mathrm e}^{2}-D_{\mathrm o}^{2}=1$ and $D'=DQ$, we deduce that 
   			\[
   			D^2-1=2DD_{\mathrm o}\quad\text{and}
   			\quad
   			\frac{tD_{\mathrm o}(t;z)}{Q(t;z)}
   			=\frac{t(D(t;z)^2-1)}{2D'(t;z)}.
   			\]
   			Set
   			\[
   			a=1+\rho,
   			\quad b=\rho-1,
   			\quad X=e^{\rho x}\quad\text{and}\quad Y=e^{\rho y}.
   			\]
   			Formula \eqref{eq:jacobi-D-closed} gives
   			\begin{align*}
   				\frac{tD_{\mathrm o}(t;y)}{Q(t;y)}
   				=\frac{t(Y^2-1)}{2\rho Y}
   				\quad\text{and}\quad D(t;x+y)-D(t;y)
   				=\frac{4\rho Y(X-1)}{(a+bXY)(a+bY)}.
   			\end{align*}
   			Since $2t=-ab$, Eq.~\eqref{eq:jacobi-F-in-D} becomes
   			\begin{align*}
   				F(t;x,y)
   				&=\frac{b+aY}{a+bY}
   				-\frac{ab(Y^2-1)(X-1)}{(a+bXY)(a+bY)}\\
   				&=\frac{(a+bY)(aY+bX)}{(a+bXY)(a+bY)}
   				=\frac{aY+bX}{a+bXY}\\
   				&=\frac{(1+\rho)e^{\rho y}+(\rho-1)e^{\rho x}}
   				{(1+\rho)+(\rho-1)e^{\rho(x+y)}},
   			\end{align*}
   			 completing the proof. 
			\end{proof}
   			
   			    \section*{Acknowledgments}
   This work was supported by the National  Science Foundation of China  (grants 12471318  \&  12322115 \& 12271301)
   and the Fundamental Research Funds for the Central Universities.


\end{document}